\title{Quaternionic Second-Order Freeness and the Fluctuations of Large Symplectically Invariant Random Matrices}
\author{C.\ E.\ I.\ Redelmeier\thanks{Research supported by a two-year Sophie Germain post-doctoral scholarship provided by the Fondation math\'{e}matique Jacques Hadarmard, held at the D\'{e}partement de Math\'{e}matiques, UMR 8628 Universit\'{e} Paris-Sud 11-CNRS, B\^{a}timent 425, Facult\'{e} des Sciences d'Orsay, Universit\'{e} Paris-Sud 11, F-91405 Orsay Cedex.}}

\documentclass[11pt]{article}

\usepackage{graphicx}
\usepackage{graphics}
\usepackage{amsfonts}
\usepackage{amssymb}
\usepackage{amsthm}
\usepackage{amsmath}
\usepackage{color}
\usepackage{cite}

\newtheorem{theorem}{Theorem}[section]
\newtheorem{lemma}[theorem]{Lemma}
\newtheorem{proposition}[theorem]{Proposition}
\newtheorem{corollary}[theorem]{Corollary}

\theoremstyle{remark}
\newtheorem{remark}[theorem]{Remark}

\newtheorem{example}[theorem]{Example}

\theoremstyle{definition}
\newtheorem{definition}[theorem]{Definition}

\begin{document}

\maketitle

\begin{abstract}
We present a definition for second-order freeness in the quaternionic case.  We demonstrate that this definition on a second-order probability space is asymptotically satisfied by independent symplectically invariant quaternionic matrices.

This definition is different from the natural definition for complex and real second-order probability spaces, those motivated by the asymptotic behaviour of unitarily invariant and orthogonally invariant random matrices respectively.

Most notably, because the quaternionic trace does not have the cyclic property of a trace over a commutative field, the asymmetries which appear in the multi-matrix context result in an asymmetric contribution from the terms which appear symmetrically in the complex and real cases.
\end{abstract}

\section{Introduction}

In \cite{MR1094052}, D.-V.\ Voiculescu established a connection between free probability, a noncommutative analogue of classical probability, and the behaviour of large random matrices.  Many ensembles of large random matrices exhibit this asymptotic freeness, a property characterizing the behaviour of the expected value of the trace of a product of independent matrices in terms of the behaviour of the matrices individually.  The same behaviour is observed in complex, real, and quaternionic random matrices.

Second-order freeness is developed in \cite{MR2216446, MR2294222, MR2302524} in order to similarly characterize the covariances of traces of multi-matrix expressions in high dimensions.  The behaviour described here is exhibited by independent ensembles of matrices whose probability distribution is unitarily invariant (unchanged if the matrices are conjugated by an arbitrary unitary matrix), a natural symmetry for complex random matrices.  The behaviour of orthogonally invariant random matrices (a natural symmetry for real matrices) is discussed in \cite{MR3217665}.  Independent orthogonally invariant random matrices do not generally satisfy the definition given in \cite{MR2216446}, so another definition is provided for real second-order freeness.  This paper gives a definition of quaternionic second-order freeness, motivated by the covariance of traces of products of large symplectically invariant matrices.

We consider the covariance of the quaternionic trace (as opposed to the real part of the trace).  While the behaviour of quaternionic eigenvalues are more complicated than eigenvalues over a commutative field (for example, the qualitatively different behaviour of left and right eigenvalues), eigenvalues may nonetheless be quaternion-valued, and the sum of quaternionic eigenvalues is related to the sum of the quaternionic diagonal entries rather than to the real part of this sum (see, e.g., \cite{MR3135396}, Chapter~9).  Thus the quaternion-valued (partial) trace is the relevant quantity.

If \(\mathbb{E}\left(\mathrm{Tr}\left(X_{k}\right)\right)=\mathbb{E}\left(\mathrm{Tr}\left(Y_{k}\right)\right)=0\), the asymptotic covariance will vanish if the number of independent terms in the two traces is different, and if the number is the same:
\begin{multline}
\label{formula: introduction}
\lim_{N\rightarrow\infty}\mathrm{cov}\left(\mathrm{Tr}\left(X_{1}\cdots X_{p}\right),\mathrm{Tr}\left(Y_{1}\cdots Y_{p}\right)\right)
\\=\prod_{i=1}^{p}\lim_{N\rightarrow\infty}\mathrm{Re}\left(\frac{1}{N}\mathrm{Tr}\left(X_{i}Y_{p+1-i}\right)\right)+\frac{1}{4}\sum_{k=1}^{p-1}\prod_{i=1}^{p}\lim_{N\rightarrow\infty}\mathrm{Re}\left(\frac{1}{N}\mathrm{Tr}\left(X_{i}Y_{p+1-k-i}\right)\right)\\-\frac{1}{2}\prod_{i=1}^{p}\lim_{N\rightarrow\infty}\mathrm{Re}\left(\frac{1}{N}\mathrm{Tr}\left(X_{i}Y_{i}^{\ast}\right)\right)+\frac{1}{4}\sum_{k=1}^{p-1}\prod_{i=1}^{p}\lim_{N\rightarrow\infty}\mathrm{Re}\left(\frac{1}{N}\mathrm{Tr}\left(X_{i}Y_{k+i}^{\ast}\right)\right)
\end{multline}
where subscripts are taken modulo \(p\) (see Definition~\ref{definition: second-order freeness} for the full definition).  This is different from the analogous expression for unitarily and orthogonally invariant matrices in two notable ways.

Most notably, because the quaternionic trace is not cyclic on products (because the quaternions are not commutative), the cyclic symmetries of traces break down in the multi-matrix context which is integral to the freeness and second-order freeness setting.  The real and complex expressions are sums over terms known as ``spoke diagrams'' (see Figure~\ref{figure: spokes}), which are invariant under cycling the terms within each of the two traces.  In the quaternionic case, two distinguished terms appear with a different coefficient.  (In the real case, all terms appear with the coefficient \(1\).  In the complex definition, the first two terms appear with coefficient \(1\), and the last two terms do not appear.)

Secondly, the contribution from each ``spoke'' is the expected value of the real part of a trace, rather than simply the expected value of the trace.  The real-part function does not appear in either the real or the complex case, where the covariance may be complex-valued when a unitarily invariant or complex-valued orthogonally invariant matrix appears (for example a Wishart matrix (see Definition~\ref{definition: Wishart}) \(G^{\ast}DG\) or \(G^{T}DG\) where the deterministic matrix \(D\) is complex-valued).  However, quaternionic conjugation is a diffeomorphism (unlike complex conjugation), so the contribution of a term and its conjugate in a symplectically symmetric distribution will be equal, and their quaternion-imaginary parts will cancel.

In this paper, we use a topological expansion for quaternionic multi-matrix expressions from \cite{2014arXiv1412.0646R} (see also \cite{MR2005857, MR2480549}).  Asymptotic terms are characterized by the annular noncrossing conditions in \cite{MR2052516} (see also \cite{MR3217665, 2012arXiv1204.6211R} for more on the unoriented case).  The necessary background is outlined in Section~\ref{section: preliminaries}.  The definitions of quaternionic second-order freeness and asymptotic quaternionic second-order freeness are also given in this section.  The asymptotic behaviour of cumulants of traces is given in Section~\ref{section: cumulants}.  The proof that independent symplectically invariant random matrices satisfy the definition of quaternionic second-order freeness is given in Section~\ref{section: second-order freeness}.  A number of important random matrix ensembles (Ginibre, GSE, Wishart, and Haar-distributed symplectic) are discussed in Section~\ref{section: zoo}, including demonstrations that they satisfy the necessary convergence conditions and some information on the values of their contributions to expression (\ref{formula: introduction}).

\section{Preliminaries}
\label{section: preliminaries}

\subsection{Partitions, permutations, and maps}

\begin{definition}
For \(m,n\in\mathbb{N}\), we let \(\left[n\right]:=\left\{1,\ldots,n\right\}\) and \(\left[m,n\right]=\left\{m,\ldots,n\right\}\).

We will sometimes append a formal point at infinity to a set: \(I_{\infty}=I\cup\left\{\infty\right\}\).

For a set \(I\subseteq\mathbb{Z}\), we let \(-I:=\left\{-k:k\in I\right\}\) and \(\pm I:=I\cup\left(-I\right)\).
\end{definition}

\begin{definition}
A {\em partition} (or a {\em set partition}) \(\pi=\left\{V_{1},\ldots,V_{m}\right\}\) of a set \(I\) is a set of nonempty subsets of \(I\) (called {\em blocks} of the partition) such that \(V_{i}\cap V_{j}=\emptyset\) for all \(i\neq j\), and \(\bigcup_{i=1}^{m}V_{i}=I\).  We denote the set of partitions of \(I\) by \({\cal P}\left(I\right)\) and the set of partitions of \(\left[n\right]\) by \({\cal P}\left(n\right)\).

For \(\pi=\left\{V_{1},\ldots,V_{m}\right\}\in{\cal P}\left(I\right)\), we let \(\pm\pi:=\left\{\pm V_{1},\ldots,\pm V_{m}\right\}\in{\cal P}\left(\pm I\right)\).  If \(J\subseteq I\), we let \(\left.\pi\right|_{J}=\left\{V_{1}\cap J,\ldots,V_{m}\cap J\right\}\in{\cal P}\left(J\right)\) (ignoring empty blocks).

We denote the number of blocks of partition \(\pi\in{\cal P}\left(I\right)\) by \(\#\left(\pi\right)\).

For a function \(f:I\rightarrow J\), we define the {\em kernel of \(f\)}, denoted \(\ker\left(f\right)\), to be the partition of \(I\) whose blocks are the nonempty preimages \(f^{-1}\left(j\right)\) for each \(j\in J\).

The partitions \({\cal P}\left(I\right)\) form a poset, where, for \(\pi,\rho\in{\cal P}\left(I\right)\), we say that \(\pi\preceq\rho\) if every block of \(\pi\) is a subset of a block of \(\rho\).

The smallest element of the poset \({\cal P}\left(I\right)\) (in which each block contains exactly one element) is denoted \(0_{I}\), and the largest element (in which all the elements are contained in one block) is denoted \(1_{I}\).  We denote \(0_{n}:=0_{\left[n\right]}\) and \(1_{n}:=1_{\left[n\right]}\), or simply \(0\) and \(1\) if \(I\) is understood.

The {\em join} of two partitions \(\pi,\rho\in{\cal P}\left(I\right)\), denoted \(\pi\vee\rho\), is the smallest element which is bigger than both \(\pi\) and \(\rho\).  The {\em meet} of \(\pi\) and \(\rho\), denoted \(\pi\wedge\rho\), is the largest partition which is smaller than both \(\pi\) and \(\rho\).  (The poset of partitions forms a lattice, so the meet and join are uniquely defined.)

We say that a partition \(\pi\in{\cal P}\left(I\right)\) {\em connects} blocks \(V,W\in\rho\) of another partition \(\rho\in{\cal P}\left(I\right)\) if \(V\) and \(W\) are subsets of the same block of \(\pi\vee\rho\).

A partition whose blocks all contain exactly \(2\) elements is called a {\em pairing}.  We denote the set of all pairings of \(I\) by \({\cal P}_{2}\left(I\right)\), and the set of all pairings of \(\left[n\right]\) by \({\cal P}_{2}\left(n\right)\).
\end{definition}

The following lemma will be useful:
\begin{lemma}
If \(\pi,\rho,\sigma\in{\cal P}\left(I\right)\), then \(\#\left(\pi\vee\rho\right)-\#\left(\pi\vee\rho\vee\sigma\right)\leq\#\left(\pi\right)-\#\left(\pi\vee\sigma\right)\).
\label{lemma: join}
\end{lemma}
\begin{proof}
If \(k\) blocks of \(\pi\vee\rho\) are joined by a block of \(\sigma\), then at least as many blocks of \(\pi\) are joined by that block.  Thus joining with a block of \(\sigma\) reduces the number of blocks of \(\pi\) by at least much as it reduces the number of blocks of \(\pi\vee\rho\).  Repeating for all the blocks of \(\sigma\), the result follows.
\end{proof}

\begin{definition}
We define an {\em integer partition} \(\lambda\) of integer \(n\geq 0\) as a list of integers (parts) \(\lambda_{1}\geq\cdots\geq\lambda_{k}>0\) such that \(\lambda_{1}+\cdots+\lambda_{k}=n\).
\end{definition}

If it is unclear from the context, we will specify whether we are referring to a set partition or an integer partition.  If it is not stated, a partition is a set partition.

\begin{definition}
We denote the set of permutations on set \(I\) by \(S\left(I\right)\) and the permutations on \(\left[n\right]\) by \(S_{n}\).

We will write permutations in cycle notation (sometimes omitting cycles with only one element).  We follow the convention that permutations are evaluated right-to-left.

We note that conjugation by another permutation substitutes each element in the cycle notation with its image under the conjugating permutation, and thus that conjugation by another permutation preserves the cycle structure of the permutation (and thus cycling the permutations in a product also preserves the cycle structure).

The orbits, or cycles, of a permutation form a partition of the domain set.  We will also write the number of cycles of a permutation \(\pi\) (that is, the number of blocks in this partition) as \(\#\left(\pi\right)\).  The domain will be specified if it is not clear from the context.  We will also use other partition notation, such as the join \(\vee\), implicitly applied to the orbits of permutations.  (This may be ambiguous in the definition of \(\mathrm{PM}_{\mathrm{nc}}\left(\pi\right)\), Definition~\ref{definition: spheres}.  While in \(\mathrm{PM}\left(\pi\right)\), \(\pi\) is interpreted as a partition, in \(\mathrm{PM}_{\mathrm{nc}}\left(\pi\right)\) it is a permutation.)

We will also consider \(\pi\in{\cal P}_{2}\left(I\right)\) to be permutations, where an element is taken to the element it is paired with.

If \(\pi\in S\left(I\right)\) and \(J\subseteq I\), we define the permutation induced by \(\pi\) on \(J\), which we denote by \(\left.\pi\right|_{J}\), by letting \(\left.\pi\right|_{J}\left(k\right)=\pi^{m}\), where \(m>0\) is the smallest integer such that \(\pi^{m}\in J\).  In cycle notation, this amounts to deleting elements which are not in \(J\).  If none of the cycles of \(\pi\) contain elements of both \(J\) and \(I\setminus J\), then this is simply the restriction of the function to \(J\).

We will frequently subscript functions by a permutation.  Let \(f\) be a function which is invariant under cycling of terms in a product (such as the trace of matrices over a commutative field).  Let \(I\subseteq\pm\left[n\right]\), and let \(\pi=\left(c_{1},\ldots,c_{n_{1}}\right)\cdots\left(c_{n_{m-1}},\ldots,c_{n_{m}}\right)\in S\left(I\right)\).  Then
\[f_{\pi}\left(x_{1},\ldots,x_{n}\right):=f\left(x_{c_{1}}\cdots x_{c_{n_{1}}}\right)\cdots f\left(x_{c_{n_{m-1}}}\cdots x_{c_{n_{m}}}\right)\textrm{.}\]
\end{definition}

Multiplying (left or right) a permutation \(\pi\) by a transposition (a permutation of the form \(\left(a,b\right)\)) increases \(\#\left(\pi\right)\) by \(1\) if \(a\) and \(b\) connect two cycles of \(\pi\), and decreases \(\#\left(\pi\right)\) by \(1\) if they do not.

We use permutations to encode graphs in surfaces (see Example~\ref{example: map}).  Faces, edges, or vertices are encoded as a permutation by enumerating the ends of the edges (or hyperedges) which appear around each face (or hyperedge or vertex) in cyclic order in a cycle of the permutation.  See \cite{MR0404045, MR2036721} for more detail on the construction, \cite{MR1813436, 2012arXiv1204.6211R} for more on the construction in the unoriented case, and \cite{MR2005857, MR2480549, 2014arXiv1412.0646R} for more on the quaternionic case.

For any surface, we may construct an orientable two-sheeted covering space by letting points be close in the covering space if they are close in the base space by an orientation preserving path.  For more on this construction, see \cite{MR1867354}, Section~3.3.  We use \(k\) and \(-k\) to represent the two preimages of \(k\).  Objects in the covering space should be consistent with the covering map, as described in the following definition:

\begin{definition}
We will denote the function \(k\mapsto -k\) by \(\delta\).

We will denote by \(\mathrm{PM}\left(I\right)\) the set of all permutations \(\pi\) on \(\pm I\) such that \(\delta\pi\delta=\pi^{-1}\) and such that no cycle of \(\pi\) contains both \(k\) and \(-k\).  We denote \(\mathrm{PM}\left(\left[n\right]\right)\) by \(\mathrm{PM}\left(n\right)\).

For \(\rho\in{\cal P}\left(I\right)\), we denote by \(\mathrm{PM}\left(\rho\right)\) the set of \(\pi\in\mathrm{PM}\left(I\right)\) such that every cycle of \(\pi\) is contained in a block of \(\pm\rho\).

If \(\pi\in\mathrm{PM}\left(I\right)\), the cycles are paired: for each cycle there is another where the elements with the signs reversed appear in reverse cyclic order.  We choose the cycle from each pair in which the element of smallest absolute value (or infinity) appears as a positive integer.  We define \(\mathrm{FD}\left(\pi\right)\) as the permutation which is the product of these cycles, taking the domain as the set of elements appearing in those cycles.  We will also use this notation to refer to the domain as a set.
\end{definition}

It is possible to calculate the permutation representing the vertices from the permutations representing the faces and the hyperedges:
\begin{definition}
If \(\varphi_{+}\in S\left(I\right)\) and \(\alpha\in\mathrm{PM}\left(I\right)\), we let \(\varphi_{-}:=\delta\varphi_{+}\delta\).  We define
\[K\left(\varphi_{+},\alpha\right):=\varphi_{+}^{-1}\alpha^{-1}\varphi_{-}\textrm{.}\]
In \cite{MR3217665} it is shown that \(K\left(\varphi_{+},\alpha\right)\in\mathrm{PM}\left(I\right)\).

Letting \(\varphi:=\varphi_{+}\varphi_{-}^{-1}\), we define the Euler characteristic:
\[\chi\left(\varphi_{+},\alpha\right):=\#\left(\varphi\right)/2+\#\left(\alpha\right)/2+\#\left(K\left(\varphi,\alpha\right)\right)/2-\left|I\right|\textrm{.}\]
We note that while \(K\left(\varphi_{+},\alpha\right)\) depends on which cycles of \(\varphi\) come from \(\varphi_{+}\) and which from \(\varphi_{-}\), the number of cycles, and hence the Euler characteristic, does not, so we may write \(\chi\left(\varphi,\alpha\right)\).
\end{definition}

\begin{example}

\begin{figure}
\centering
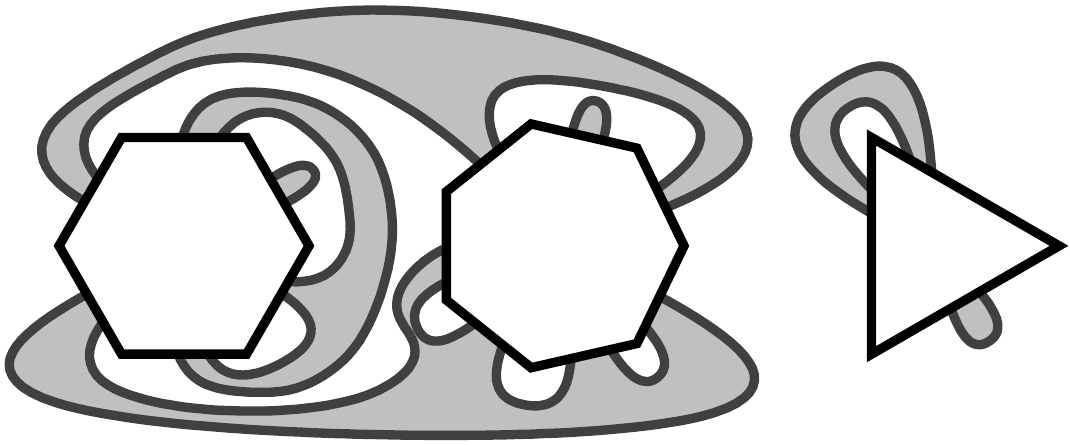
\caption{A hypermap on three faces.  The middle face is shown from behind.}
\label{figure: map}
\end{figure}

In Figure~\ref{figure: map}, an example of a hypermap is shown.  The middle face is shown from behind, so the integers are negative and appear in reversed (clockwise) order.  The faces are represented by the permutation
\[\varphi_{+}=\left(1,2,3,4,5,6\right)\left(7,8,9,10,11,12,13\right)\left(14,15,16\right)\textrm{.}\]
Permutation \(\varphi_{+}\) gives us the fronts of the faces, while permutation
\begin{multline*}
\varphi_{-}^{-1}:=\delta\varphi_{+}\delta\\=\left(-6,-5,-4,-3,-2,-1\right)\left(-13,-12,-11,-10,-9,-8,-7\right)\left(-16,-15,-14\right)
\end{multline*}
gives the backs of the faces.

The hyperedges are represented by permutation
\begin{multline*}
\alpha=\left(1\right)\left(-1\right)\left(2,6,5\right)\left(-5,-6,-2\right)\left(3,-11,-13\right)\left(13,11,-3\right)\\\left(4,-7,-8,-9,-10\right)\left(10,9,8,7,-4\right)\left(12\right)\left(-12\right)\left(14,15\right)\left(-15,-14\right)\left(16\right)\left(-16\right)\textrm{.}
\end{multline*}
(We use the convention that all elements (faces, hyperedges, vertices) are oriented counter-clockwise internally, which is slightly different from the conventions used in some of the cited references.)

The integer representing the corner of a face appearing in a given vertex is the integer labelling the edge of which it is the counter-clockwise vertex (so the clockwise vertex if the face is viewed from behind).  We calculate that the vertices will be represented by permutation
\begin{multline*}
K\left(\varphi_{+},\alpha\right)=\varphi_{+}^{-1}\alpha^{-1}\varphi_{-}
\\=\left(1,6\right)\left(-6,-1\right)\left(2,4,-10\right)\left(10,-4,-2\right)\left(3,-13\right)\left(13,-3\right)\left(5\right)\left(-5\right)\left(7\right)\left(-7\right)\\\left(8\right)\left(-8\right)\left(9\right)\left(-9\right)\left(11,12\right)\left(-12,-11\right)\left(14\right)\left(-14\right)\left(15,16\right)\left(-16,-15\right)\textrm{.}
\end{multline*}

We compute that the Euler characteristic is
\begin{multline*}
\chi\left(\varphi,\alpha\right)=\#\left(\varphi\right)/2+\#\left(\alpha\right)/2+\#\left(K\left(\varphi,\alpha\right)\right)/2-\left|I\right|
\\=3+7+10-16=4\textrm{.}
\end{multline*}
This is consistent with the intuition that the diagram is noncrossing, and thus each of the two connected components is a sphere.
\label{example: map}
\end{example}

The Euler characteristic is at most \(2\) times the number of connected components:
\begin{lemma}
\label{lemma: spheres}
Let \(\varphi,\alpha\in\mathrm{PM}\left(I\right)\).  Then
\[\chi\left(\varphi,\alpha\right)\leq 2\#\left(\pm\varphi\vee\alpha\right)\textrm{.}\]
\end{lemma}
See \cite{MR3217665} for a proof.

\begin{definition}
Let \(\varphi_{+}\in S\left(I\right)\), with \(\varphi\in\mathrm{PM}\left(I\right)\) defined as above.  We denote the set of \(\alpha\in\mathrm{PM}\left(I\right)\) such that \(\chi\left(\varphi_{+},\alpha\right)=2\#\left(\pm\varphi_{+}\vee\alpha\right)\) by \(\mathrm{PM}_{\mathrm{nc}}\left(\varphi_{+}\right)\) or by \(\mathrm{PM}_{\mathrm{nc}}\left(\varphi\right)\).
\label{definition: spheres}
\end{definition}

In some contexts, we may characterize \(\alpha\in\mathrm{PM}_{\mathrm{nc}}\left(\varphi\right)\) in terms of crossing conditions.  In particular, we may do so when \(\#\left(\varphi_{+}\right)\) is equal to \(1\) or \(2\).

\begin{definition}
Let \(\varphi,\alpha\in S\left(I\right)\) where \(\#\left(\varphi\right)=1\).

We say that \(\alpha\) is {\em disc nonstandard} on \(\varphi\) if there are distinct \(a,b,c\in I\) such that \(\left.\varphi\right|_{\left\{a,b,c\right\}}=\left.\alpha\right|_{\left\{a,b,c\right\}}=\left(a,b,c\right)\).

We say that \(\alpha\) is {\em disc crossing} on \(\varphi\) if there are distinct \(a,b,c,d\in I\) such that \(\left.\varphi\right|_{\left\{a,b,c,d\right\}}=\left(a,b,c,d\right)\) and \(\left.\alpha\right|_{\left\{a,b,c,d\right\}}=\left(a,c\right)\left(b,d\right)\).

We say that \(\alpha\) is {\em disc noncrossing} on \(\varphi\) if it is neither disc nonstandard or disc crossing.  We denote the set of \(\alpha\in S\left(I\right)\) which are disc noncrossing on \(\varphi\) by \(S_{\mathrm{disc-nc}}\left(\varphi\right)\).
\end{definition}

\begin{definition}
\label{definition: annular noncrossing}
Let \(\varphi,\alpha\in S\left(I\right)\), where \(\#\left(\varphi\right)=2\).  We say that \(\alpha\) is {\em annular nonstandard} on \(\varphi\) if it satisfies at least one of the following two conditions:
\begin{enumerate}
  \item There exist distinct elements \(a,b,c\in I\) such that \(\left.\varphi\right|_{\left\{a,b,c\right\}}=\left.\alpha\right|_{\left\{a,b,c\right\}}=\left(a,b,c\right)\).
  \item There exist distinct elements \(a,b,c,d\in I\) such that \(\left.\varphi\right|_{\left\{a,b,c,d\right\}}=\left(a,b\right)\left(c,d\right)\) and \(\left.\alpha\right|_{\left\{a,b,c,d\right\}}=\left(a,c,b,d\right)\).
\end{enumerate}

For \(x,y\in I\) in distinct cycles of \(\varphi\), we define
\[\lambda_{x,y}:=\left(\varphi\left(x\right),\varphi^{2}\left(x\right)\ldots,\varphi^{-1}\left(x\right),\varphi\left(y\right),\varphi^{2}\left(y\right),\ldots,\varphi^{-1}\left(y\right)\right)\in S\left(I\setminus\left\{x,y\right\}\right)\textrm{.}\]
We say that \(\alpha\) is {\em annular crossing} on \(\varphi\) if it satisfies at least one of the following three conditions:
\begin{enumerate}
  \item There exist distinct elements \(a,b,c,d\in I\) such that \(\left.\varphi\right|_{\left\{a,b,c,d\right\}}=\left(a,b,c,d\right)\) and \(\left.\alpha\right|_{\left\{a,b,c,d\right\}}=\left(a,c\right)\left(b,d\right)\).
  \item There exist distinct elements \(a,b,c,x,y\in I\) such that \(x\) and \(y\) are in distinct orbits of \(\varphi\) and such that \(\left.\lambda_{x,y}\right|_{\left\{a,b,c\right\}}=\left(a,b,c\right)\) and \(\left.\alpha\right|_{\left\{a,b,c,x,y\right\}}=\left(a,b,c\right)\left(x,y\right)\).
  \item There exist distinct elements \(a,b,c,d,x,y\in I\) where \(x\) and \(y\) belong to distinct cycles of \(\varphi\) and \(\left.\lambda_{x,y}\right|_{\left\{a,b,c,d\right\}}=\left(a,b,c,d\right)\) and \(\left.\alpha\right|_{\left\{a,b,c,d,x,y\right\}}=\left(a,c\right)\left(b,d\right)\left(x,y\right)\).
\end{enumerate}

We say that \(\alpha\) is {\em annular noncrossing} on \(\varphi\) if it does not satisfy any of these conditions.  We denote the set of \(\alpha\in S\left(I\right)\) which are annular noncrossing on \(\varphi\) by \(S_{\mathrm{ann-nc}}\left(\varphi\right)\).
\end{definition}

\begin{lemma}
Let \(\varphi_{+}\in S\left(I\right)\) with \(\#\left(\varphi_{+}\right)=1\), \(\varphi\in\mathrm{PM}\left(I\right)\) as above, and let \(\alpha\in\mathrm{PM}\left(I\right)\).  Then \(\chi\left(\varphi,\alpha\right)=2\) if and only if \(\alpha\) does not connect the cycles of \(\varphi\), and \(\left.\alpha\right|_{I}\) is disc noncrossing on \(\varphi_{+}\).

Let \(\varphi_{+}\in S\left(I\right)\) with \(\#\left(\varphi\right)=2\), and let \(\alpha\in\mathrm{PM}\left(I\right)\) such that \(\pm\varphi\vee\alpha=1_{\pm I}\).  Then \(\chi\left(\varphi,\alpha\right)=2\) if and only if \(\alpha\) connects one orbit \(J\) of \(\varphi\) to exactly one other orbit of \(\varphi\) which is not \(-J\) (call it \(K\)), and \(\left.\alpha\right|_{J\cup K}\) is annular noncrossing on \(\left.\varphi\right|_{J\cup K}\).
\end{lemma}

See \cite{2012arXiv1204.6211R} for the proofs.  For more on the disc-noncrossing conditions, see, e.g. \cite{MR2266879, MR0404045, MR2036721}.  The noncrossing conditions for the annulus in the orientable case are established in \cite{MR2052516}.

The following lemma can be thought of as the additivity of the Euler characteristic over disconnected components and the computation of its value on connected sums.

\begin{lemma}
Let \(\pi=\left\{V_{1},\ldots,V_{m}\right\}\in{\cal P}\left(I\right)\) be a partition of \(I\) with \(m\) blocks, and let \(\alpha\in\mathrm{PM}\left(\pi\right)\).

Let \(\varphi_{+}\in S\left(I\right)\).  If \(\varphi_{+}\preceq\pi\), then
\[\chi\left(\varphi,\alpha\right)=\chi\left(\left.\varphi\right|_{V_{1}},\left.\alpha\right|_{V_{1}}\right)+\cdots+\chi\left(\left.\varphi\right|_{V_{m}},\left.\alpha\right|_{V_{m}}\right)\textrm{.}\]
If \(\varphi_{+}\) has one cycle, then
\[\chi\left(\varphi,\alpha\right)=\chi\left(\left.\varphi\right|_{V_{1}},\left.\alpha\right|_{V_{1}}\right)+\cdots+\chi\left(\left.\varphi\right|_{V_{m}},\left.\alpha\right|_{V_{m}}\right)-2\left(m-1\right)\textrm{.}\]
\label{lemma: chi}
\end{lemma}
\begin{proof}
If \(\varphi_{+}\preceq\pi\), then every cycle of \(K\left(\varphi,\alpha\right)\) is also contained within a block of \(\pi\).  The first part of the result follows.

Let \(\varphi_{0}\) be a permutation and let \(\tau=\left(a,b\right)\in S\left(I\right)\) be a transposition conncting two cycles of \(\varphi_{0}\) and two blocks of \(\pm\alpha\).  Then \(\left(a,b\right)\) connects two cycles of \(K\left(\varphi_{0},\alpha\right)\) (as above), so left multiplying \(\left(a,b\right)\) and right multiplying by \(\left(-a,-b\right)\) both reduce the number of cycles, that is \(\chi\left(\varphi_{0}\tau,\alpha\right)=\chi\left(\varphi_{0},\alpha\right)-2\).  Likewise, if \(\tau\) does not connect cycles of \(\varphi_{0}\), but connects blocks of \(\pm\alpha\), then \(\chi\left(\varphi_{0}\tau,\alpha\right)=\chi\left(\varphi_{0},\alpha\right)+2\).  Thus, if we can construct \(\varphi_{+}\) from \(\varphi_{0}\) by right-multiplication by a series of transpositions \(\tau_{1},\ldots,\tau_{m}\in S_{n}\) that do not connect blocks of \(\pm\alpha\), then
\begin{equation}
\label{formula: connected sum}
\chi\left(\varphi_{+},\alpha\right)=\chi\left(\varphi_{0},\alpha\right)+2\left(\#\left(\varphi_{+}\right)-\#\left(\varphi_{0}\right)\right)\textrm{.}
\end{equation}
Let \(\varphi_{+}\in S\left(I\right)\) have a single cycle.  We note that \(\varphi_{+}^{\prime}:=\left.\varphi_{+}\right|_{V_{1}}\cdots\left.\varphi_{+}\right|_{V_{m}}\) differs from \(\varphi_{+}\) only on \(k\in I\) where \(k\) belongs to a different block of \(\pi\) from \(\varphi_{+}\left(k\right)\), so \(\varphi_{+}^{\prime}\varphi_{+}^{-1}\) consists of single element cycles as well as cycles \(\left(c_{1},\ldots,c_{m}\right)\) where \(c_{i}\) always belongs to a different block of \(\pi\) than \(c_{i+1}\).  We can write such a cycle \(\left(c_{1},\ldots,c_{m}\right)=\left(c_{1},c_{2}\right)\left(c_{2},c_{3}\right)\cdots\left(c_{m-1},c_{m}\right)\) as a product of transpositions containing elements in different blocks of \(\pi\).  Then since \(K\left(\varphi_{+}^{\prime},\alpha\right)=\varphi_{+}^{\prime}\varphi_{+}^{-1}K\left(\varphi_{+},\alpha\right)\varphi_{-}\varphi_{-}^{\prime-1}\), the second part of the result follows from (\ref{formula: connected sum}).
\end{proof}

\subsection{Quaternions}

\begin{definition}
The {\em quaternions} \(\mathbb{H}\) are an algebra \(a+bi+cj+dk\), \(a,b,c,d\in\mathbb{R}\), with \(i,j,k\) satisfying \(ij=k\), \(jk=i\), and \(ki=j\).

The real part of a quaternion is
\[\mathrm{Re}\left(a+bi+cj+dk\right)=a\textrm{.}\]

The quaternion conjugate is
\[\overline{a+bi+cj+dk}=a-bi-cj-dk\textrm{.}\]
\end{definition}

The quaternions may be faithfully represented as \(2\times 2\) complex matrices
\[\left(\begin{array}{cc}a+bi&c+di\\-c+di&a-bi\end{array}\right)\textrm{.}\]
We will use \(1\) and \(-1\) to index the rows and columns.  We note that the real part is the normalized trace, and the quaternion conjugate is the (complex) conjugate transpose of the matrix.

\subsection{Matrices}

We will often use subscripts on matrices, \(A_{k}\).  We will often move the subscript to a superscript, in brackets: \(A^{\left(k\right)}\).  We will use a bracketed superscript of \(1\) to represent the matrix itself, and a bracketed superscript of \(-1\) to represent the (quaternion) conjugate transpose of the matrix.  These notations may be combined: \(A^{\left(-k\right)}=A_{k}^{\ast}\).

\begin{definition}
We denote the usual trace by \(\mathrm{Tr}\left(A\right)=\sum_{i=1}^{n}A_{ii}\) and the normalized trace by \(\mathrm{tr}=\frac{1}{n}\mathrm{Tr}.\)
\end{definition}

\begin{definition}
A bracket diagram is a product in which subexpressions may be enclosed in brackets.

We define the permutation \(\pi\in\mathrm{PM}\left(I_{\infty}\right)\) of a bracket diagram on terms with subscripts in \(I\) by placing a term with subscript \(\infty\) before and after the expression (outside any brackets).  For each \(k\in I_{\infty}\), let \(\pi\left(k\right)\) be the subscript of the term after the one subscripted \(k\), ignoring bracketed intervals.
\end{definition}

\begin{definition}
Let \(\pi,\rho\in\mathrm{PM}\left(\left[n\right]_{\infty}\right)\).  Then we define a quaternion-valued matrix in terms of its entries:
\begin{multline*}
\left[\mathrm{Re}_{\mathrm{FD}\left(\pi\right)}\mathrm{tr}_{\mathrm{FD}\left(\rho\right)}\left(A_{1},\ldots,A_{n}\right)\right]_{i_{\infty},i_{\rho\left(\infty\right)};h_{\infty},h_{\pi\left(\infty\right)}}
\\:=2^{-\left(\#\left(\pi\right)-1\right)}N^{-\left(\#\left(\rho\right)-1\right)}\sum_{\substack{i:\mathrm{FD}\left(\pi\right)\setminus\left\{\infty,\pm\rho\left(\infty\right)\right\}\rightarrow\left[N\right]\\h:\mathrm{FD}\left(\pi\right)\setminus\left\{\infty,\pi\left(\infty\right)\right\}\rightarrow\left\{1,-1\right\}}}\prod_{k\in\mathrm{FD}\left(\pi\right)\setminus\left\{\infty\right\}}A^{\left(k\right)}_{i_{k},i_{\rho\left(k\right)};h_{k},h_{\pi\left(k\right)}}\textrm{.}
\end{multline*}
\end{definition}
We note that there is only one \(\rho\in\mathrm{PM}\left(\left[n\right]_{\infty}\right)\) corresponding to a given \(\mathrm{FD}\left(\rho\right)\), so \(\rho\left(k\right)\) is well defined, even if \(k\) is not in \(\mathrm{FD}\left(\rho\right)\).  See \cite{2014arXiv1412.0646R} for how this generalizes the more standard trace along a permutation, as well as the proof of the following lemma:
\begin{lemma}
Consider a bracket diagram on a product of \(X_{k}\), \(k\in I\), where \(\mathrm{Re}\) and \(\mathrm{tr}\) are applied to bracketed intervals.  Let \(\pi\) be the permutation of the bracket diagram with only the brackets to which \(\mathrm{Re}\) is applied, and let \(\rho\) be the permutation of the bracket diagram with only the brackets to which \(\mathrm{tr}\) is applied.  Then the value of the bracket diagram is \(\mathrm{Re}_{\mathrm{FD}\left(\pi\right)}\mathrm{tr}_{\mathrm{FD}\left(\rho\right)}\left(X_{1},\ldots,X_{n}\right)\).
\end{lemma}

Thus, for example
\[X_{1}\mathrm{tr}\left(\mathrm{Re}\left(X_{2}^{\ast}\mathrm{Re}\left(X_{3}\right)\right)X_{4}^{\ast}\mathrm{Re}\left(X_{5}\right)\right)X_{6}=\mathrm{Re}_{\pi}\mathrm{tr}_{\rho}\left(X_{1},X_{2},X_{3},X_{4},X_{5},X_{6}\right)\]
where \(\pi=\left(\infty,1,-4,6\right)\left(-2\right)\left(3\right)\left(5\right)\) and \(\left(\infty,1,6\right)\left(-2,3,-4,5\right)\).

\subsection{Moments and cumulants}

Throughout, \(\left(\Omega,{\cal F},\mathbb{P}\right)\) will be a probability triple.

We will give a definition of cumulants which may be applied to quaternion-valued random variables.  Since the variables do not commute, we wish to take the necessary expected values before we perform the multiplications.

\begin{definition}
The M\"{o}bius function (on the poset of partitions) is a function \(\mu:{\cal P}\left(I\right)^{2}\rightarrow\mathbb{C}\)
which is given by
\[\mu\left(\pi,\rho\right):=\left\{\begin{array}{ll}\prod_{V\in\rho}\left(-1\right)^{\left|\left\{U\in\pi:U\subseteq V\right\}\right|}\left(\left|\left\{U\in\pi:U\subseteq V\right\}\right|\right)!\textrm{,}&\pi\preceq\rho\\0\textrm{,}&\textrm{otherwise}\end{array}\right.\textrm{.}\]

The M\"{o}bius function has the property that for any \(\pi,\rho\in{\cal P}\left(I\right)\) with \(\pi\preceq\rho\):
\[\sum_{\sigma\in{\cal P}\left(I\right):\pi\preceq\sigma\preceq\rho}\mu\left(\pi,\sigma\right)=\sum_{\sigma\in{\cal P}\left(I\right):\pi\preceq\sigma\preceq\rho}\mu\left(\sigma,\rho\right)=\left\{\begin{array}{ll}1\textrm{,}&\pi=\rho\\0\textrm{,}&\textrm{otherwise}\end{array}\right.\textrm{.}\]
\end{definition}

See \cite{MR2266879}, Chapter~10, for more on M\"{o}bius functions in general and on the computation of this particular M\"{o}bius function.

\begin{definition}
Let \(Q_{1},\ldots,Q_{n}:\Omega\rightarrow\mathbb{H}\) be random variables.  The {\em \(n\)th mixed moment} is the function
\[a_{n}\left(Q_{1},\ldots,Q_{n}\right)=\mathbb{E}\left(Q_{1}\cdots Q_{n}\right)\textrm{.}\]

We define a moment associated to each \(\pi=\left\{V_{1},\ldots,V_{m}\right\}\in{\cal P}\left(n\right)\).  For each \(V_{k}\in\pi\), we define a probability space \(\left(\Omega_{k},{\cal F}_{k},\mathbb{P}_{k}\right)\).  For each \(i\in V_{k}\), we define random variables \(Q^{\prime}_{i}:\Omega_{k}\rightarrow\mathbb{H}\) which are jointly distributed identically to the \(Q_{i}:\Omega\rightarrow\mathbb{H}\).  The moment associated to \(\pi\) is the expected value on the product of the probability spaces:
\[a_{\pi}\left(Q_{1},\ldots,Q_{n}\right):=\int_{\Omega_{1}\times\cdots\times\Omega_{m}}Q^{\prime}_{1}\cdots Q^{\prime}_{n}d\mathbb{P}_{1}\cdots d\mathbb{P}_{m}\]
(so the \(Q_{k}^{\prime}\) from different blocks are independent).

The cumulant associated with \(\pi\in{\cal P}\left(n\right)\) is the function
\[k_{\pi}:=\sum_{\rho\in{\cal P}\left(n\right):\rho\preceq\pi}\mu\left(\rho,\pi\right)a_{\rho}\textrm{.}\]
The cumulants then satisfy the {\em moment-cumulant formula}:
\[a_{\pi}=\sum_{\rho\in{\cal P}\left(n\right):\rho\preceq\pi}k_{\rho}\textrm{.}\]
We define
\[k_{n}:=k_{1_{n}}=\sum_{\pi\in{\cal P}\left(n\right)}\mu\left(\pi,1_{n}\right)a_{\pi}\textrm{.}\]
\label{definition: cumulants}
\end{definition}
If a cumulant has an argument which is independent from all the others, then the cumulant will vanish.  We note that \(k_{2}\) is the covariance of its arguments.  See \cite{MR1669968} for more on quaternionic covariance.

If a quaternionic random variable is sufficiently symmetric, then its expected value is real (since quaternionic conjugation is a diffeomorphism which preserves volume).  All of the quantities we will be dealing with have this property: we can see that the expected value of any expression calculated with Proposition~\ref{proposition: topological expansion} will be real.

\begin{lemma}
Let \(Q_{1},\ldots,Q_{n}:\Omega\rightarrow\mathbb{H}\) be quaternionic random variables with \(\mathbb{E}\left(Q_{i_{1}}^{\left(\varepsilon\left(1\right)\right)}\cdots Q_{i_{m}}^{\left(\varepsilon\left(n\right)\right)}\right)\) real for any \(i_{1}\,\ldots,i_{m}\in\left[n\right]\).  For \(\pi\in{\cal P}\left(n\right)\), let \(\zeta=\left(\infty,1,\ldots,n\right)\), and let \(\zeta^{\prime}=\prod_{V\in\pi}\left.\zeta\right|_{V}\).  Then
\[a_{\pi}\left(Q_{1},\ldots,Q_{n}\right)=\left(\mathbb{E}\circ\mathrm{Re}\right)_{\zeta^{\prime}}\left(Q_{1},\ldots,Q_{n}\right)\]
and, denoting the smallest element of \(V\) by \(v_{0}\):
\[k_{\pi}\left(Q_{1},\ldots,Q_{n}\right)=\prod_{V\in\rho}k_{\left|V\right|}\left(v_{0},\zeta^{\prime}\left(v_{0}\right),\ldots,\zeta^{\prime\left(\left|V\right|-1\right)}\left(v_{0}\right)\right)\textrm{.}\]
\label{lemma: quaternion cumulant}
\end{lemma}
\begin{proof}
We use the notation from Definition~\ref{definition: cumulants}.  If we have integrated all but one of the probability spaces (say, all except for \(\Omega_{k}\)), then all variables \(Q^{\prime}_{i}\) with \(i\notin V_{k}\) have been integrated and are therefore real and commute with the \(Q^{\prime}_{i}\) with \(i\in V_{k}\).  Thus if \(V_{k}=\left\{i_{1},\ldots,i_{m}\right\}\) with \(i_{1}<\cdots<i_{m}\), the contribution of the \(\Omega_{k}\) is \(\mathbb{E}\left(Q_{i_{1}}\cdots Q_{i_{m}}\right)\), that is, \(\mathbb{E}\left(\mathrm{Re}_{\left.\zeta\right|_{V_{k}}}\left(Q_{1},\ldots,Q_{n}\right)\right)\) (since the expected value is real), giving the first part of the result.

Noting that the M\"{o}bius function \(\mu\left(\pi,\rho\right)\) is multiplicative over the blocks of \(\rho\), we may calculate:
\[k_{\rho}\left(Q_{1},\ldots,Q_{n}\right)=\sum_{\pi\preceq\rho}\prod_{V\in\rho}\mu\left(\left.\pi\right|_{V},\left\{V\right\}\right)\left(\mathbb{E}\circ\mathrm{Re}\right)_{\left.\zeta^{\prime}\right|_{V}}\left(Q_{1},\ldots,Q_{n}\right)\]
which gives the second part.
\end{proof}

\begin{example}
We calculate that if \(Q_{1},Q_{2},Q_{3},Q_{4}\) satisfy the hypotheses of Lemma~\ref{lemma: quaternion cumulant}, then
\begin{multline*}
k_{4}\left(Q_{1},Q_{2},Q_{3},Q_{4}\right)\\=\mathbb{E}\left(Q_{1}Q_{2}Q_{3}Q_{4}\right)-\mathbb{E}\left(Q_{1}\right)\mathbb{E}\left(Q_{2}Q_{3}Q_{4}\right)-\mathbb{E}\left(Q_{2}\right)\mathbb{E}\left(Q_{1}Q_{3}Q_{4}\right)\\-\mathbb{E}\left(Q_{3}\right)\mathbb{E}\left(Q_{1}Q_{2}Q_{4}\right)-\mathbb{E}\left(Q_{4}\right)\mathbb{E}\left(Q_{1}Q_{2}Q_{3}\right)-\mathbb{E}\left(Q_{1}Q_{2}\right)\mathbb{E}\left(Q_{3}Q_{4}\right)\\-\mathbb{E}\left(Q_{1}Q_{3}\right)\mathbb{E}\left(Q_{2}Q_{4}\right)-\mathbb{E}\left(Q_{1}Q_{4}\right)\mathbb{E}\left(Q_{2}Q_{3}\right)+2\mathbb{E}\left(Q_{1}Q_{2}\right)\mathbb{E}\left(Q_{3}\right)\mathbb{E}\left(Q_{4}\right)\\+2\mathbb{E}\left(Q_{1}Q_{3}\right)\mathbb{E}\left(Q_{2}\right)\mathbb{E}\left(Q_{4}\right)+2\mathbb{E}\left(Q_{1}Q_{4}\right)\mathbb{E}\left(Q_{2}\right)\mathbb{E}\left(Q_{3}\right)\\+2\mathbb{E}\left(Q_{2}Q_{3}\right)\mathbb{E}\left(Q_{1}\right)\mathbb{E}\left(Q_{4}\right)+2\mathbb{E}\left(Q_{2}Q_{4}\right)\mathbb{E}\left(Q_{1}\right)\mathbb{E}\left(Q_{3}\right)\\+2\mathbb{E}\left(Q_{3}Q_{4}\right)\mathbb{E}\left(Q_{1}\right)\mathbb{E}\left(Q_{2}\right)-6\mathbb{E}\left(Q_{1}\right)\mathbb{E}\left(Q_{2}\right)\mathbb{E}\left(Q_{3}\right)\mathbb{E}\left(Q_{4}\right)\textrm{.}
\end{multline*}
We note that the real part function is redundant, since the expected values of the products are all real.  However, it is used in the expression in order to give the order of terms in the products.  (It would also be possible to use the trace function in its place.)
\end{example}

\subsection{Topological Expansion}

The Weingarten function is defined in \cite{MR1959915, MR2217291, MR2567222}, and the results in the lemma that follows are proven there.  (See \cite{2014arXiv1412.0646R} for more detail on the quaternionic case.)  The following definition is the quaternionic Weingarten function.  (Throughout, the Weingarten function will refer to the quaternionic Weingarten function.  It may be calculated from the real Weingarten function: \(\mathrm{Wg}^{\mathrm{Sp}\left(N\right)}=\left(-1\right)^{n/2}\mathrm{Wg}^{O\left(-2N\right)}\).  See \cite{MR2217291, MR2567222} for tables of values.)

\begin{definition}
For \(\pi,\rho\in{\cal P}_{2}\left(n\right)\), let \(d\left(\pi,\rho\right):=n/2-\#\left(\pi\vee\rho\right)\).

Define matrix with rows and columns indexed by \({\cal P}_{2}\left(n\right)\) by letting \(\mathrm{Gr}\left(\pi,\rho\right)=\left(-1\right)^{n/2}\left(-2N\right)^{\#\left(\pi\vee\rho\right)}\) for \(\pi,\rho\in{\cal P}_{2}\left(n\right)\).  We define the (quaternionic) Weingarten function and normalized Weingarten function \(\mathrm{Wg},\mathrm{wg}:{\cal P}_{2}^{2}\rightarrow\mathbb{R}\):
\[\mathrm{Wg}\left(\pi,\rho\right):=\mathrm{Gr}^{-1}\left(\pi,\rho\right)\textrm{;}\]
\[\mathrm{wg}\left(\pi,\rho\right):=\left(-2N\right)^{n/2+d\left(\pi,\rho\right)}\mathrm{Wg}\left(\pi,\rho\right)\textrm{.}\]

We note that the Weingarten function depends only on the block structure of \(\pi\vee\rho\).  Since all the blocks contain an even number of elements, we define an integer partition \(\Lambda\left(\pi\vee\rho\right)\) whose parts \(\lambda_{1},\ldots,\lambda_{k}\) are half the sizes of the blocks of \(\pi\vee\rho\).  If \(\lambda\) is an integer partition of \(n\), we define \(\mathrm{wg}\left(\lambda\right):=\mathrm{wg}\left(\pi,\rho\right)\) for any \(\pi,\rho\) such that \(\Lambda\left(\pi\vee\rho\right)=\lambda\).

For \(\pi_{1},\pi_{2}\in{\cal P}_{2}\left(n\right)\) with \(\pi_{1}\vee\pi_{2}=\pi\in{\cal P}\left(n\right)\) and \(\rho,\sigma\in{\cal P}\left(n\right)\) with \(\pi\preceq\rho\preceq\sigma\), we define a sort of cumulant of the Weingarten function by
\[c_{\pi,\rho,\sigma}:=\sum_{\tau:\rho\preceq\tau\preceq\sigma}\mu\left(\tau,\sigma\right)\prod_{V\in\tau}\mathrm{wg}\left(\left.\pi_{1}\right|_{V},\left.\pi_{2}\right|_{V}\right)\textrm{,}\]
or equivalently,
\[\prod_{V\in\sigma}\mathrm{wg}\left(\left.\pi_{1}\right|_{V},\left.\pi_{2}\right|_{V}\right)=\sum_{\tau:\rho\preceq\tau\preceq\sigma}c_{\pi,\rho,\tau}\textrm{.}\]
\end{definition}

The following is from \cite{MR1959915, MR2217291}, where more details can be found:
\begin{lemma}[Collins, \'{S}niady]
The Weingarten function may be expressed:
\begin{multline*}
\mathrm{wg}\left(\pi,\rho\right)
\\=\left(2N\right)^{d\left(\pi,\rho\right)}\sum_{k\geq 0}\left(-1\right)^{k}\sum_{\substack{\sigma_{0},\ldots,\sigma_{k}\in{\cal P}_{2}\left(n\right)\\\sigma_{0}=\pi,\sigma_{k}=\rho\\\sigma_{0}\neq\sigma_{1}\neq\cdots\neq\sigma_{k}}}\left(2N\right)^{-\left(d\left(\sigma_{0},\sigma_{1}\right)+\cdots+d\left(\sigma_{k-1},\sigma_{k}\right)\right)}\textrm{.}
\end{multline*}

For \(\pi,\rho\in{\cal P}_{2}\left(n\right)\),
\[\lim_{N\rightarrow\infty}\mathrm{wg}\left(\pi,\rho\right)=\prod_{i\in\Lambda\left(\pi,\rho\right)}\left(-1\right)^{i-1}C_{i-1}+O\left(N^{-1}\right)\]
where \(C_{k}:=\frac{1}{k+1}\binom{2k}{k}\) is the \(k\)th Catalan number.

For \(\pi_{1},\pi_{2}\in{\cal P}_{2}\left(n\right)\) with \(\pi_{1}\vee\pi_{2}=\pi\in{\cal P}\left(n\right)\) and \(\rho,\sigma\in{\cal P}\left(n\right)\) such that \(\pi\preceq\rho\preceq\sigma\), the cumulant \(c_{\pi,\rho,\sigma}\) is \(O\left(N^{2\left(\#\left(\sigma\right)-\#\left(\rho\right)\right)}\right)\).
\label{lemma: Weingarten}
\end{lemma}
\begin{proof}
The form for the Weingarten function and its asymptotics may be found in \cite{MR2217291}.

We define functional \(S\) on functions on the positive integers \(\mathbb{N}^{\ast}\rightarrow\mathbb{C}\) which are eventually zero by \(S\left(f\right)=\sum_{k=1}^{\infty}\left(-1\right)^{k}f\left(k\right)\) and a product \(g\ast h\) by
\[\left(g\ast h\right)\left(k\right)=\sum_{U_{1},U_{2}\subseteq\left[k\right]:U_{1}\cup U_{2}=\left[k\right]}g\left(\left|U_{1}\right|\right)h\left(\left|U_{2}\right|\right)\textrm{.}\]
The product is associative, and \(S\left(g\ast h\right)=S\left(g\right)S\left(h\right)\) (see \cite{MR1959915}).

Let \(\tau=\left\{V_{1},\ldots,V_{t}\right\}\in{\cal P}\left(n\right)\).  For \(m\in\left[t\right]\) and for each nonnegative integer \(l\), define \(f_{m}^{\left(l\right)}:\mathbb{N}\rightarrow\mathbb{C}\) by letting \(f_{m}^{\left(l\right)}\left(k\right)\) be the number paths in \({\cal P}_{2}\left(V_{m}\right)\) (distance metric \(d\)) of \(k\) steps and total length \(l\) (where each step is distinct from the previous one).  Treating \(\mathrm{wg}\) as a function of \(2N\), the coefficient of \(\left(2N\right)^{d\left(\pi_{1},\pi_{2}\right)-l}\) in \(\prod_{V\in\tau}\mathrm{wg}\left(\left.\pi_{1}\right|_{V},\left.\pi_{2}\right|_{V}\right)\) is
\begin{multline*}
\sum_{l_{1},\cdots,l_{t}:l_{1}+\cdots+l_{t}=l}S\left(f_{1}^{\left(l_{1}\right)}\right)\cdots S\left(f_{t}^{\left(l_{t}\right)}\right)
\\=\sum_{l_{1},\cdots,l_{t}:l_{1}+\cdots+l_{t}=l}S\left(f_{1}^{\left(l_{1}\right)}\ast\cdots\ast f_{t}^{\left(l_{t}\right)}\right)
\end{multline*}
which may be interpreted as the coefficient of \(\left(2N\right)^{d\left(\pi_{1},\pi_{2}\right)-l}\) in
\[\left(2N\right)^{d\left(\pi_{1},\pi_{2}\right)}\sum_{k\geq 0}\left(-1\right)^{k}\sum_{\substack{\upsilon_{0},\ldots,\upsilon_{k}\in{\cal P}_{2}\left(\tau\right)\\\upsilon_{0}=\pi_{1},\upsilon_{k}=\pi_{2}\\\upsilon_{0}\neq\upsilon_{1}\neq\cdots\neq\upsilon_{k}}}\left(2N\right)^{-\left(d\left(\upsilon_{0},\upsilon_{1}\right)+\cdots+d\left(\upsilon_{k-1},\upsilon_{k}\right)\right)}\]
(where the \(U_{m}\) from the definition of the product \(\ast\) are interpreted as the sets on which the \(k_{m}\) steps of the path on \(V_{m}\) take place).

The cumulant \(c_{\pi,\rho,\sigma}\) is the sum over \(\tau\) with \(\rho\preceq\tau\preceq\sigma\) of \(\mu\left(\tau,\sigma\right)\) times the above sum.  If we consider all occurrences of the term associated with path \(\upsilon_{0},\ldots,\upsilon_{k}\), it will occur in the sum associated with all \(\tau\) with \(\upsilon_{0}\vee\cdots\vee\upsilon_{k}\preceq\tau\).  The sum of the coefficients will be \(\sum_{\tau:\rho,\upsilon_{0}\vee\cdots\vee\upsilon_{k}\preceq\tau\preceq\sigma}\mu\left(\tau,\sigma\right)\), which vanishes unless \(\rho\vee\upsilon_{0}\vee\cdots\vee\upsilon_{k}=\sigma\).  Thus the cumulant may be expressed
\[c_{\pi,\rho,\sigma}=\left(2N\right)^{d\left(\pi_{1},\pi_{2}\right)}\sum_{k\geq 0}\left(-1\right)^{k}\sum_{\substack{\upsilon_{0},\ldots,\upsilon_{k}\in{\cal P}_{2}\left(\tau\right)\\\upsilon_{0}=\pi_{1},\upsilon_{k}=\pi_{2}\\\upsilon_{0}\neq\upsilon_{1}\neq\cdots\neq\upsilon_{k}\\\rho\vee\upsilon_{0}\vee\cdots\vee\upsilon_{k}=\sigma}}\left(2N\right)^{-\left(d\left(\upsilon_{0},\upsilon_{1}\right)+\cdots+d\left(\upsilon_{k-1},\upsilon_{k}\right)\right)}\textrm{.}\]

Let \(\upsilon_{0},\ldots,\upsilon_{k}\) be a path in \({\cal P}_{2}\left(n\right)\) with \(\upsilon_{0}\vee\cdots\vee\upsilon_{k}=\rho\) of total length \(l\).  Joining \(\upsilon\in{\cal P}_{2}\left(n\right)\) to \(\upsilon_{m}\) must reduce the number of blocks by at least as much as joining \(\upsilon\) to \(\rho\) (Lemma~\ref{lemma: join}).  Thus \(d\left(\upsilon,\upsilon_{m}\right)=n/2-\#\left(\upsilon\vee\upsilon_{m}\right)\geq\#\left(\rho\right)-\#\left(\upsilon\vee\rho\right)\), so inserting it into the path must increase the total length of the path by at least \(2\#\left(\left(\rho\right)-\#\left(\upsilon\vee\rho\right)\right)\).  Thus the length of the path from \(\pi_{1}\) to \(\pi_{2}\) must be increased by at least \(2\left(\#\left(\pi\right)-\#\left(\sigma\right)\right)\geq 2\left(\#\left(\rho\right)-\#\left(\sigma\right)\right)\).  The result follows.
\end{proof}

The following definition is adapted from in \cite{MR2337139, MR2483727}.  See also \cite{MR2240781}.

\begin{definition}
Let \(X_{1},\ldots,X_{n}:\Omega\rightarrow M_{N\times N}\left(\mathbb{H}\right)\) be random quaternionic matrices.  Then the {\em normalized symplectic matrix cumulant} is the function \(f:\mathrm{PM}\left(n\right)\rightarrow\mathbb{C}\)
\begin{multline*}
f\left(\alpha\right)=\sum_{\pi\in\mathrm{PM}\left(n\right)}\left(-2N\right)^{\chi\left(\alpha,\pi\right)-\#\left(\alpha\right)}\mathrm{wg}\left(\delta\alpha,\delta\pi\right)
\\\times\mathbb{E}\left(\mathrm{Re}_{\mathrm{FD}\left(\pi\right)}\mathrm{tr}_{\mathrm{FD}\left(\pi\right)}\left(X_{1},\ldots,X_{n}\right)\right)\textrm{.}
\end{multline*}
\end{definition}

We note that if the matrices have a limit distribution, then
\begin{multline*}
\lim_{N\rightarrow\infty}f\left(\alpha\right)
\\=\sum_{\substack{\pi\in\mathrm{PM}_{\mathrm{nc}}\left(\alpha\right)\\\pi\preceq\alpha}}\prod_{i\in\Lambda\left(\left(\delta\alpha\right)\vee\left(\delta\pi\right)\right)}\left(-1\right)^{i-1}C_{i-1}\left(\lim_{N\rightarrow\infty}\circ\mathbb{E}\circ\mathrm{Re}\circ\mathrm{tr}\right)\left(X_{1},\ldots,X_{n}\right)
\end{multline*}
which is multiplicative in the sense that for \(\tau=\left\{V_{1},\ldots,V_{m}\right\}\in{\cal P}\left(n\right)\), if \(\alpha\preceq\pm\tau\), then, denoting the matrix cumulant of \(X_{i}\), \(i\in V_{k}\) associated with \(\left.\alpha\right|_{V_{k}}\) by \(f_{k}\left(\left.\alpha\right|_{V_{k}}\right)\),
\[\lim_{N\rightarrow\infty}f\left(\alpha\right)=\lim_{N\rightarrow\infty}f_{1}\left(\left.\alpha\right|_{V_{1}}\right)\cdots\lim_{N\rightarrow\infty}f_{m}\left(\left.\alpha\right|_{V_{m}}\right)\textrm{.}\]
The cumulants are multiplicative over independent matrices (see \cite{2014arXiv1412.0646R}):
\begin{lemma}
Let \(w:\left[n\right]\rightarrow\left[C\right]\) be a word in colours \(\left[C\right]\), and let \(X_{1},\ldots,X_{n}:\Omega\rightarrow M_{N\times N}\left(\mathbb{H}\right)\) be symplectically invariant random quaternionic matrices.  For each \(c\in\left[C\right]\), let the \(X_{k}\) with \(w\left(k\right)=c\) be independent from the rest of the matrices, and have normalized matrix cumculants \(f_{c}:\mathrm{PM}\left(w^{-1}\left(c\right)\right)\rightarrow\mathbb{C}\).  Then the normalized matrix cumulants \(f:\mathrm{PM}\left(n\right)\rightarrow\mathbb{C}\) vanish unless \(\alpha\preceq\pm\ker\left(w\right)\), in which case
\[f\left(\alpha\right)=f_{1}\left(\left.\alpha\right|_{w^{-1}\left(1\right)}\right)\cdots f_{C}\left(\left.\alpha\right|_{w^{-1}\left(C\right)}\right)\textrm{.}\]
\end{lemma}
Integrals of traces may be computed in terms of the matrix cumulants.  See \cite{2014arXiv1412.0646R}:
\begin{proposition}
Let \(X_{1},\ldots,X_{n}:\Omega\rightarrow\mathbb{H}\) be symplectically invariant matrices with normalized matrix cumulants \(f:\mathrm{PM}\left(\left[n\right]\right)\rightarrow\mathbb{C}\).  Let \(\varphi_{\mathrm{Re}},\varphi_{\mathrm{tr}}\in\left[n\right]_{\infty}\).  Let \(\varepsilon:\left[n\right]\rightarrow\left\{1,-1\right\}\), and let \(\delta_{\varepsilon}\in S\left(\pm\left[n\right]\right):k\mapsto\varepsilon\left(\left|k\right|\right)k\).  Then
\begin{multline*}
\mathbb{E}\left[\mathrm{Re}_{\varphi_{\mathrm{Re}}}\mathrm{tr}_{\varphi_{\mathrm{tr}}}\left(X_{1}^{\left(\varepsilon\left(1\right)\right)},\ldots,X_{n}^{\left(\varepsilon\left(n\right)\right)}\right)\right]
\\=\sum_{\alpha\in\mathrm{PM}\left(n\right)}\left(-2\right)^{\chi\left(\varphi_{\mathrm{Re}},\delta_{\varepsilon}\alpha\delta_{\varepsilon}\right)-2\#\left(\varphi_{\mathrm{Re}}\right)}N^{\chi\left(\varphi_{\mathrm{tr}},\delta_{\varepsilon}\alpha\delta_{\varepsilon}\right)-2\#\left(\varphi_{\mathrm{tr}}\right)}f\left(\alpha\right)
\end{multline*}
where \(\alpha\) acts trivially on \(\pm\infty\).
\label{proposition: topological expansion}
\end{proposition}

\subsection{Freeness}

\begin{definition}
A {\em noncommutative probability space} is a pair \(\left(A,\phi_{1}\right)\) consisting of a unital algebra \(A\) and a linear functional \(\phi_{1}:A\rightarrow\mathbb{F}\) such that \(\phi_{1}\left(1_{A}\right)=1\).  (In this paper, \(\mathbb{F}\) will generally be \(\mathbb{H}\), although \(\mathbb{F}=\mathbb{C}\) and \(\mathbb{F}=\mathbb{R}\) are more frequently studied.)

Subalgebras \(A_{1},\ldots,A_{C}\subseteq A\) are {\em free} if \(\phi_{1}\left(x_{1}\cdots x_{p}\right)=0\) whenever, for all \(k\in\left[p\right]\), \(\phi_{1}\left(x_{k}\right)=0\) and \(x_{k}\in A_{w\left(k\right)}\) where \(w:\left[p\right]\rightarrow\left[C\right]\) is a word with \(w\left(1\right)\neq w\left(2\right)\neq\cdots\neq w\left(p\right)\).

Sets are free if the algebras they generate are free.  (If \(A\) is a \(\ast\)-algebra, the algebra generated by a set should also be closed under the operation \(\ast\).)
\end{definition}

\begin{definition}
If we have a set of random \(N\times N\) matrices \(\left\{X_{\lambda}\right\}_{\lambda\in\Lambda}\) for arbitrarily large \(N\in\mathbb{Z}\) (we will usually suppress the dimension in the notation), then we say they have a {\em limit distribution} \(\left\{x_{\lambda}\right\}_{\lambda\in\Lambda}\subseteq A\), where \(\left(A,\phi_{1}\right)\) is a noncommutative probability space, if, for any \(\lambda_{1},\ldots,\lambda_{p}\in\Lambda\):
\[\lim_{N\rightarrow\infty}\mathbb{E}\left(\mathrm{tr}\left(X_{\lambda_{1}}\cdots X_{\lambda_{p}}\right)\right)=\phi_{1}\left(x_{\lambda_{1}}\cdots x_{\lambda_{p}}\right)\]
and all higher cumulants of normalized traces of products of the \(X_{\lambda}\) vanish for large \(N\).

We may also wish the limit distribution to respect a structure such as an involution \(\ast:A\rightarrow A\), \(a\mapsto a^{\ast}\).  We incorporate this into the subscript notation: for each \(\lambda\in\Lambda\), we must have \(-\lambda\in\Lambda\), and \(X_{\lambda}^{\ast}=X_{-\lambda}\) (where here \(\ast\) is the conjugate transpose) and \(x_{\lambda}^{\ast}=x_{-\lambda}\).

In the quaternionic context, we would like the involution to be conjugate-linear and reverse the order of multiplication: \(\left(qx\right)^{\ast}=\overline{q}x^{\ast}\) and \(\left(xy\right)^{\ast}=y^{\ast}x^{\ast}\), for \(q\in\mathbb{H}\) and \(x,y\in A\).

Sets of matrices are {\em asymptotically free} if their limits in \(A\) are free.
\end{definition}

Second-order probability spaces are constructed in \cite{MR2216446, MR2294222, MR2302524} in order to model the covariance of traces, in order to study the fluctuations around the limits (central limit type results, as opposed to law of large numbers type results).

\begin{definition}
A {\em second-order probability space} is a triple \(\left(A,\phi_{1},\phi_{2}\right)\) where \(\left(A,\phi_{1}\right)\) is a noncommutative probability space and \(\phi_{2}:A^{2}\rightarrow\mathbb{F}\) is a bilinear function.  (In the complex and real cases, \(\phi_{2}\) is required to be tracial in each of its arguments.)

A quaternionic second-order probability space must also be a \(\ast\)-algebra, where \(\ast\) is an involution \(A\rightarrow A\), \(a\mapsto a^{\ast}\), which is conjugate linear and reverses the order of multiplication.  (In the complex case, the involution is not required, and in the real case the involution is linear.)

We say that subalgebras \(A_{1},\ldots,A_{C}\subseteq A\) are {\em quaternion second-order free} if they are free; and letting \(u:\left[p\right]\rightarrow\left[C\right]\) be a word with either \(u\left(1\right)\neq u\left(2\right)\neq\cdots\neq u\left(p\right)\neq u\left(1\right)\) or \(p=1\) and \(v:\left[q\right]\rightarrow\left[C\right]\) be a word with either \(v\left(1\right)\neq v\left(2\right)\neq\cdots\neq v\left(q\right)\neq v\left(1\right)\) or \(q=1\), when \(p\neq q\) or if \(p=q=1\) and \(u\left(1\right)\neq v\left(1\right)\):
\[\phi_{2}\left(x_{1}\cdots x_{p},y_{1}\cdots y_{q}\right)=0\]
and when \(p=q>2\):
\begin{multline*}
\phi_{2}\left(x_{1}\cdots x_{p},y_{1}\cdots y_{q}\right)
\\=\prod_{i=1}^{p}\mathrm{Re}\left(\phi_{1}\left(x_{i}y_{p+1-i}\right)\right)+\frac{1}{4}\sum_{k=1}^{p-1}\prod_{i=1}^{p}\mathrm{Re}\left(\phi_{1}\left(x_{i}y_{p+1-k-i}\right)\right)\\-\frac{1}{2}\prod_{i=1}^{p}\mathrm{Re}\left(\phi_{1}\left(x_{i}y_{i}^{\ast}\right)\right)+\frac{1}{4}\sum_{k=1}^{p-1}\prod_{i=1}^{p}\mathrm{Re}\left(\phi_{1}\left(x_{i}y_{k+i}^{\ast}\right)\right)\textrm{.}
\end{multline*}
(All indices are taken modulo the appropriate value to fall in the appropriate interval.)

We note that if the subalgebras are quaternion second-order free, then if \(x_{i}\) and \(y_{j}\) do not come from the same subalgebra, \(\phi_{1}\left(x_{i}y_{j}\right)=\phi_{1}\left(x_{i}y_{j}^{\ast}\right)=0\) (since they are free) so any term containing such an expression will vanish.

Sets are quaternion second-order free if the \(\ast\)-algebras they generate are quaternion second-order free.
\label{definition: second-order freeness}
\end{definition}

The expression for the value of \(\phi_{2}\) is motivated by Proposition~\ref{proposition: spokes}.  For the expression in the complex case, see \cite{MR2216446, MR2294222, MR2302524}, and for the expression in the real case, see \cite{MR3217665}.  Like the complex and real cases, the terms correspond to spoke diagrams (see Figure~\ref{figure: spokes}), but in the quaternionic case they appear with different weights and signs.

\begin{figure}
\label{figure: spokes}
\centering
\begin{tabular}{cc}
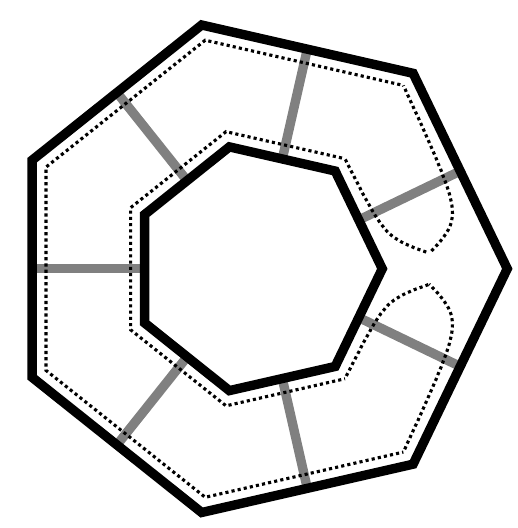&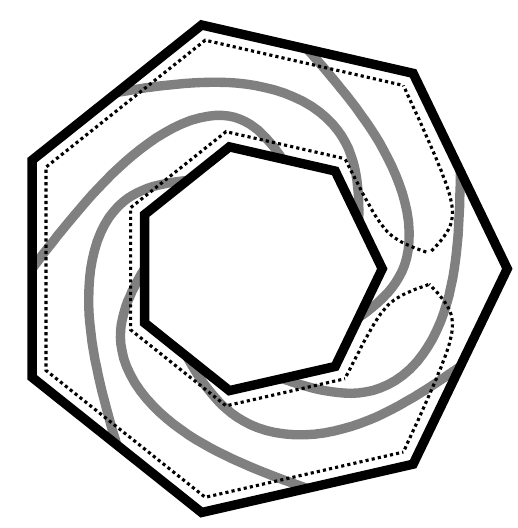\\
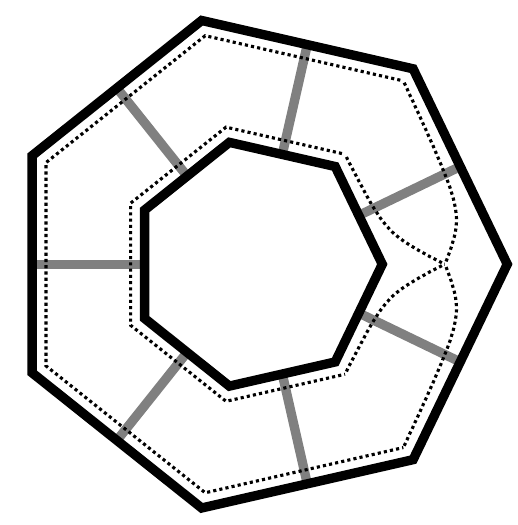&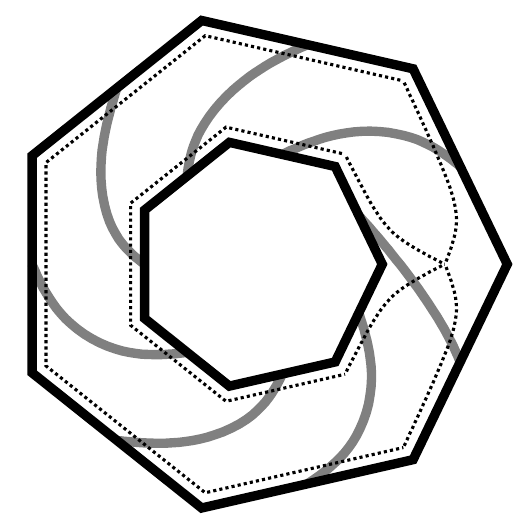
\end{tabular}
\caption{The four types of spoke diagrams which appear in the quaternionic case.  The solid polygons represent the cycles of \(\mathrm{tr}\), and the dotted lines the cycle of \(\mathrm{Re}\).  The upper left is a sphere (\(k=p\) where \(k\) is the signed offset defined in Lemma~\ref{lemma: spokes}, \(\chi=2\)), the upper right a torus (\(k=2\), \(\chi=0\)), the lower left a projective plane (\(k=-1\), \(\chi=1\)), and the lower right a Klein bottle (\(k=-2\), \(\chi=0\)).}
\end{figure}

\begin{definition}
We say that matrices \(\left\{X_{\lambda}\right\}_{\lambda\in\Lambda}\) have a {\em second-order limit distribution} \(\left\{x_{\lambda}\right\}_{\lambda\in\Lambda}\subseteq A\) if this set is a limit distribution for \(\left(A,\phi_{1}\right)\), and in addition
\begin{multline*}
\lim_{N\rightarrow\infty}k_{2}\left(\mathrm{Tr}\left(X_{\lambda_{1}}\cdots X_{\lambda_{p}}\right),\mathrm{Tr}\left(X_{\lambda_{p+1}}\cdots X_{\lambda_{p+q}}\right)\right)
\\=\phi_{2}\left(x_{\lambda_{1}}\cdots x_{\lambda_{p}},x_{\lambda_{p+1}}\cdots x_{\lambda_{p+q}}\right)\textrm{,}
\end{multline*}
and all higher cumulants of unnormalized traces vanish for large \(N\).

In the quaternionic case, we require that \(A\) has a conjugate-linear involution \(\ast:A\rightarrow A\), \(x\mapsto x^{\ast}\), and for each \(\lambda\in\Lambda\), \(-\lambda\in\Lambda\), and \(X_{-\lambda}=X_{\lambda}^{\ast}\) and \(x_{-\lambda}=x_{\lambda}^{\ast}\).

We say that matrices are {\em asymptotically quaternion second-order free} if they have a second-order limit distribution and their images in the limit distribution are second-order free.
\end{definition}

We note that if matrices have a limit distribution, then if we expand a product of traces in cumulants:
\[\lim_{N\rightarrow\infty}a_{\pi}\left(\mathrm{tr}\left(X_{1}\right),\ldots,\mathrm{tr}\left(X_{n}\right)\right)=\lim_{N\rightarrow\infty}\sum_{\rho\preceq\pi}\prod_{V\in\rho}k_{\left|V\right|}\left(\mathrm{tr}\left(X_{i_{1}}\right),\ldots,\mathrm{tr}\left(X_{i_{\left|V\right|}}\right)\right)\textrm{.}\]
Since all cumulants \(k_{m}\) with \(m>1\) vanish as \(N\rightarrow\infty\), the only remaining term is
\[\lim_{N\rightarrow\infty}\mathbb{E}\left(\mathrm{tr}\left(X_{1}\right)\right)\cdots\mathbb{E}\left(\mathrm{tr}\left(X_{n}\right)\right)\textrm{.}\]
If they have a second-order limit distribution, then the difference from this value is \(O\left(N^{-1}\right)\).

\section{Cumulants}

\label{section: cumulants}

In this section we show a sense in which the cumulants of traces of products of symplectically invariant random matrices are sums over connected terms, and look at bounds on the order of cumulants of traces of expressions in independent matrices with limit distributions.

\begin{proposition}
\label{proposition: cumulants}
Let \(0<n_{1}<\cdots<n_{m}\), and let \(X_{1},\ldots,X_{n_{m}}:\Omega\rightarrow M_{N\times N}\left(\mathbb{H}\right)\) be symplectically invariant random quaternionic matrices.  Let \(w:\left[n_{m}\right]\rightarrow\left[C\right]\) be a word in colours \(\left[C\right]\), and let the matrices \(X_{k}\) with \(w\left(k\right)=c\) be independent from the other matrices, with matrix cumulants \(f_{c}:\mathrm{PM}\left(w^{-1}\left(c\right)\right)\rightarrow\mathbb{C}\).

Let
\[\zeta=\left(\infty,1,\ldots,n_{m}\right)\]
and let
\[\varphi=\left(\infty\right)\left(1,\ldots,n_{1}\right)\cdots\left(n_{m-1}+1,\ldots,n_{m}\right)\textrm{.}\]
Then, denoting the cycles of \(\mathrm{FD}\left(\pi\right)\) by \(\pi_{1},\ldots,\pi_{p}\), and for \(\sigma\in{\cal P}\left(p\right)\), letting \(\hat{\sigma}=\left\{\bigcup_{k\in V}\pi_{k}:V\in\sigma\right\}\):
\begin{multline*}
k_{n}\left(\mathrm{tr}\left(X_{1}\cdots X_{n_{1}}\right),\ldots,\mathrm{tr}\left(X_{n_{m-1}+1}\cdots X_{n_{m}}\right)\right)
\\=\sum_{\substack{\alpha\in\mathrm{PM}\left(n\right),\pi\in \mathrm{PM}\left(n\right)\\\rho:\left(\delta\alpha\right)\vee\left(\delta\pi\right)\preceq\rho\preceq\pm\ker\left(w\right)\\\sigma\in{\cal P}\left(p\right):\hat{\sigma}\preceq\pm\ker\left(w\right)\\\pm\varphi\vee\alpha\vee\rho\vee\hat{\sigma}=1_{\pm\left[n_{m}\right]}}}\left(-2\right)^{\chi\left(\zeta,\alpha\right)-2+\chi\left(\alpha,\pi\right)-\#\left(\alpha\right)}N^{\chi\left(\varphi,\alpha\right)-2m+\chi\left(\alpha,\pi\right)-\#\left(\alpha\right)}\\\times c_{\left(\delta\alpha\right)\vee\left(\delta\pi\right),\left(\delta\alpha\right)\vee\left(\delta\pi\right),\rho}k_{\sigma}\left(\left(\mathrm{Re}\circ\mathrm{tr}\right)_{\pi_{1}}\left(X_{1},\ldots,X_{n_{m}}\right),\ldots,\left(\mathrm{Re}\circ\mathrm{tr}\right)_{\pi_{p}}\left(X_{1},\ldots,X_{n_{m}}\right)\right)\textrm{.}
\end{multline*}
\end{proposition}
\begin{proof}
For \(\tau\in{\cal P}\left(m\right)\), we define \(\hat{\tau}\in{\cal P}\left(n_{m}\right)\) by \(\hat{\tau}:=\left\{\bigcup_{k\in V}\left[n_{k-1}+1,n_{k}\right]:V\in\tau\right\}\).  Using Proposition~\ref{proposition: topological expansion}, we compute the summand in the cumulant corresponding to \(\tau\):
\begin{multline*}
\mu\left(\tau,1_{m}\right)\sum_{\alpha\in\mathrm{PM}\left(\hat{\tau}\wedge\ker\left(w\right)\right)}\left(-2\right)^{\chi\left(\zeta,\alpha\right)-2}N^{\chi\left(\varphi,\alpha\right)-2m}
\\\times\sum_{\pi\in\mathrm{PM}\left(\hat{\tau}\wedge\ker\left(w\right)\right)}\left(-2N\right)^{\chi\left(\alpha,\pi\right)-\#\left(\alpha\right)}\left(\prod_{U\in\hat{\tau}\wedge\ker\left(w\right)}\mathrm{wg}\left(\left.\delta\alpha\right|_{\pm U},\left.\delta\pi\right|_{\pm U}\right)\right)
\\\times\left(\prod_{V\in\hat{\tau}\wedge\ker\left(w\right)}\mathbb{E}\left(\left(\mathrm{Re}\circ\mathrm{tr}\right)_{\mathrm{FD}\left(\left.\pi\right|_{\pm V}\right)}\left(X_{1},\ldots,X_{n_{m}}\right)\right)\right)\textrm{.}
\end{multline*}
(The Euler characteristics are calculated using Lemma~\ref{lemma: chi}: \(\sum_{V\in\hat{\tau}}\chi\left(\varphi,\alpha\right)=\chi\left(\varphi,\alpha\right)\) and \(\sum_{V\in\hat{\tau}}\left(\chi\left(\left.\zeta\right|_{V},\alpha\right)-2\right)=\chi\left(\zeta,\alpha\right)-2\).)

We collect terms associated to a fixed \(\alpha\in\mathrm{PM}\left(\ker\left(w\right)\right)\) (there will be one for each \(\tau\in{\cal P}\left(m\right)\) such that \(\pm\hat{\tau}\succeq\alpha\)).  We expand the Weingarten function and expected value in cumulants:
\[\prod_{U\in\hat{\tau}\wedge\ker\left(w\right)}\mathrm{wg}\left(\left.\delta\alpha\right|_{\pm U},\left.\delta\pi\right|_{\pm U}\right)=\sum_{\rho:\left(\delta\alpha\right)\vee\left(\delta\pi\right)\preceq\rho\preceq\pm\hat{\tau}\wedge\ker\left(w\right)}c_{\left(\delta\alpha\right)\vee\left(\delta\pi\right),\left(\delta\alpha\right)\vee\left(\delta\pi\right),\rho}\]
(so all \(\rho\preceq\pm\ker\left(w\right)\), and a term associated with \(\rho\) appears in the expansion for \(\tau\) if \(\pm\hat{\tau}\succeq\rho\)).  We likewise expand the expected value in cumulants:
\begin{multline*}
\prod_{V\in\hat{\tau}\wedge\ker\left(w\right)}\mathbb{E}\left(\left(\mathrm{Re}\circ\mathrm{tr}\right)_{\mathrm{FD}\left(\left.\pi\right|_{\pm V}\right)}\left(X_{1},\ldots,X_{n_{m}}\right)\right)
\\=\sum_{\substack{\sigma\in{\cal P}\left(p\right)\\\pm\hat{\sigma}\preceq\pm\hat{\tau}\wedge\ker\left(w\right)}}k_{\sigma}\left(\left(\mathrm{Re}\circ\mathrm{tr}\right)_{\pi_{1}}\left(X_{1},\ldots,X_{n_{m}}\right),\ldots,\left(\mathrm{Re}\circ\mathrm{tr}\right)_{\pi_{p}}\left(X_{1},\ldots,X_{n_{m}}\right)\right)
\end{multline*}
(so all \(\sigma\) have \(\hat{\sigma}\preceq\pm\ker\left(w\right)\) and the term associated with \(\sigma\) appears in the expansion for \(\tau\) if \(\pm\hat{\tau}\succeq\hat{\sigma}\)).  Collecting terms associated to a fixed \(\alpha\), \(\rho\), and \(\sigma\), the sum of the M\"{o}bius functions (noting that \(\mu\left(\tau,1_{m}\right)=\mu\left(\pm\hat{\tau},1_{\pm\left[n_{m}\right]}\right)\)) is
\[\sum_{\tau:\pm\hat{\tau}\succeq\pm\varphi\vee\alpha\vee\rho\vee\hat{\sigma}}\mu\left(\pm\hat{\tau},1_{\pm\left[n_{m}\right]}\right)\textrm{,}\]
which vanishes unless \(\pm\varphi\vee\alpha\vee\rho\vee\hat{\sigma}=1_{\pm\left[n_{m}\right]}\).
\end{proof}

\begin{corollary}
\label{corollary: cumulants}
If independent sets of random matrices have limit distributions, the the \(\ast\)-algebra they generate also has a limit distribution, and if they have second-order limit distributions, then the \(\ast\)-algebra they generate also has a second-order limit distribution.

If \(m>1\), then the \(m\)th cumulant is \(O\left(N^{-m}\right)\), and in the notation of Proposition~\ref{proposition: cumulants}, terms of order \(N^{-m}\) have \(\alpha\in\mathrm{PM}_{\mathrm{nc}}\left(\varphi\right)\) and \(\pi\in\mathrm{PM}_{\mathrm{nc}}\left(\alpha\right)\).
\end{corollary}
\begin{proof}
By Lemma~\ref{lemma: spheres}, we know that \(\chi\left(\varphi,\alpha\right)\leq 2\#\left(\pm\varphi\vee\alpha\right)\) and \(\chi\left(\alpha,\pi\right)\leq 2\#\left(\pm\alpha\vee\pi\right)\), so \(\chi\left(\varphi,\alpha\right)+\chi\left(\alpha,\pi\right)\leq 2\#\left(\pm\varphi\vee\alpha\vee\pi\right)+2\#\left(\pm\alpha\right)\) (Lemma~\ref{lemma: join}).  Similarly, \(\left(\delta\alpha\right)\vee\left(\delta\pi\right)\preceq\pm\varphi\vee\alpha\vee\pi\), so \(\chi\left(\varphi,\alpha\right)+\chi\left(\alpha,\pi\right)+2\#\left(\rho\right)-2\#\left(\left(\delta\alpha\right)\vee\left(\delta\pi\right)\right)\leq 2\#\left(\pm\varphi\vee\alpha\vee\pi\vee\rho\right)+2\#\left(\pm\alpha\right)\).  Thus if \(m=1\), the limit exists, so the matrices have a limit distribution

Since the \(X_{k}\), \(w\left(k\right)=c\) have a second-order limit distribution, then if block \(V\in\sigma\) has \(\left|V\right|>1\), the cumulant associated with the block \(V\) is \(O\left(N^{-\left|V\right|}\right)\), and \(\lim_{N\rightarrow\infty}N^{\left|V\right|}k_{\left|V\right|}\) vanishes for \(\left|V\right|>2\).  (Mixed cumulants always vanish, since the terms are independent.)  In a nonzero term, \(\pm\varphi\vee\alpha\vee\rho\vee\hat{\sigma}=1_{\pm\left[n_{m}\right]}\), so if \(\pm\varphi\vee\alpha\vee\pi\vee\rho\) has more than one block, each block must contain at least one \(\pi_{k}\) contained in a block of \(\hat{\sigma}\) containing more than one element.  We calculate that if \(m>1\), the \(m\)th cumulant is \(O\left(N^{-m}\right)\).  Furthermore, if \(m>2\), then either a block of \(\sigma\) with more than one element has more than two elements, in which case the limit of \(N^{m}\) times the cumulant vanishes, or the blocks of \(\sigma\) with more than one element all have two elements, in which case more than \(\#\left(\varphi\vee\alpha\vee\pi\vee\rho\right)\) elements are contained in such blocks, and the cumulant is \(O\left(N^{-m^{\prime}}\right)\) for some \(m^{\prime}>m\).  Thus the generated algebra has a second-order limit distribution.

A term from the cumulant which does not vanish faster than \(N^{-m}\) (or \(N^{0}\) when \(m=1\)) must have \(\chi\left(\varphi,\alpha\right)=2\#\left(\varphi\vee\alpha\right)\) and \(\chi\left(\alpha,\pi\right)=2\#\left(\alpha,\pi\right)\), which demonstrates the second part.
\end{proof}

\begin{example}
The diagram shown in Figure~\ref{figure: map} corresponds to an \(\alpha\in\mathrm{PM}_{\mathrm{nc}}\left(\varphi\right)\), as would appear in the expansion of 
\[k_{3}\left(\mathrm{tr}\left(X_{1}X_{2}X_{3}X_{4}X_{5}X_{6}\right),\mathrm{tr}\left(X_{7}X_{8}X_{9}X_{10}X_{11}X_{12}X_{13}\right),\mathrm{tr}\left(X_{14}X_{15}X_{16}\right)\right)\textrm{.}\]
Since \(\pm\varphi\vee\alpha=\left\{\pm\left[13\right],\pm\left[14,16\right]\right\}\), either \(\pi\preceq\alpha\) or \(\pi\) connects at least one cycle of \(\alpha\) from \(\pm\left[13\right]\) to one cycle from from \(\pm\left[14,16\right]\), as well as their paired cycles (so \(\chi\left(\alpha,\pi\right)=\#\left(\alpha\right)-2\)).  In the first case, either \(\rho\) or \(\sigma\) must connect \(\pm\left[13\right]\) and \(\pm\left[14,16\right]\).

We calculate
\begin{multline*}
K\left(\zeta_{+},\alpha\right)\\=\left(1,16,15,13,-3,6\right)\left(-6,3,-13,-15,-16,-1\right)\left(2,4,-10\right)\left(10,-4,-2\right)\\\left(5\right)\left(-5\right)\left(7\right)\left(-7\right)\left(8\right)\left(-8\right)\left(9\right)\left(-9\right)\left(11,12\right)\left(-12,-11\right)\left(14\right)\left(-14\right)\textrm{,}
\end{multline*}
so \(\chi\left(\zeta,\alpha\right)=0\), contributing a factor of \(\frac{1}{4}\).
\end{example}

\section{Second-Order Freeness}
\label{section: second-order freeness}

\begin{lemma}
Let \(p>1\), let \(\varphi_{+}=\left(1,\ldots,p\right)\in S_{p}\), let \(\varphi_{-}=\left(-1,\ldots,-p\right)\), and let \(\varphi=\varphi_{+}\varphi_{-}^{-1}\).  Let \(\alpha\in\mathrm{PM}_{\mathrm{nc}}\left(\varphi\right)\) have no cycle containing both \(k\) and \(k+1\), \(k\in\left[p-1\right]\).  Then \(\alpha\) has at least one cycle with only one element.  If, in addition, \(p\) is not in the same cycle of \(\alpha\) as \(1\), then \(\alpha\) must have at least two cycles with only one element.
\label{lemma: singlets}
\end{lemma}
\begin{proof}
We can find an \(a\in\left[p\right]\) such that \(\alpha^{-1}\left(a\right)>a\) (if all elements of \(\left[p\right]\) are all fixed points, the result follows, and if \(\alpha^{-1}\left(a\right)<a\), then \(\alpha^{-1}\left(\alpha^{-1}\left(a\right)\right)>\alpha^{-1}\left(a\right)\) since \(\alpha\) is disc standard).  Then \(\alpha^{-1}\left(a\right)\neq\varphi\left(a\right)\), so there are \(\varphi\left(a\right),\ldots,\varphi^{k}\left(a\right)\) which do not share an orbit with \(a\) or any other elements of \(\left[p\right]\) (by the disc-nonstandard and disc-noncrossing conditions, respectively).  If any of these is not a fixed point, then we may repeat the argument with a \(k^{\prime}<k\).  Thus there must be at least one fixed point.

If \(p\) and \(1\) are not in the same orbit of \(\alpha\), then we may repeat the argument taking \(\varphi\left(a\right)\) as \(a\), which gives us a distinct set of elements, so we may find at least one more fixed point.
\end{proof}

\begin{proposition}
Independent sets of symplectically invariant quaternionic random matrices with a limit distribution are asymptotically free.
\label{proposition: free}
\end{proposition}
\begin{proof}
Let \(A_{1},\ldots,A_{C}\) be independent sets of symplectically invariant quaternionic random matrices, and let \(X_{1},\ldots,X_{p}:\Omega\rightarrow\mathbb{H}\) have \(\mathbb{E}\left(\mathrm{tr}\left(X_{k}\right)\right)=0\) and \(X_{k}\in A_{w\left(k\right)}\), where \(w:\left[p\right]\rightarrow\left[C\right]\) is a word where \(w\left(1\right)\neq w\left(2\right)\neq\cdots\neq w\left(p\right)\).  We consider the value of
\[\lim_{N\rightarrow\infty}\mathrm{tr}\left(X_{1}\cdots X_{p}\right)\textrm{.}\]
According to Proposition~\ref{proposition: cumulants} (using the same notation) when \(m=1\), surviving terms will have \(\alpha\in\mathrm{PM}\left(\varphi\right)\) and \(\pi\preceq\alpha\) (since \(\chi\left(\alpha,\pi\right)\leq 2\#\left(\pm\alpha\vee\pi\right)\)).  By Lemma~\ref{lemma: singlets}, there is a cycle of \(\alpha\), and hence of \(\pi\), which contains only one element, so a term \(\mathbb{E}\left(\mathrm{Re}\left(\mathrm{tr}\left(X_{k}\right)\right)\right)=0\) appears for some \(k\in\left[p\right]\).  Thus the expression vanishes.
\end{proof}

\begin{lemma}
Let
\[\varphi_{+}=\left(1,\ldots,p\right)\left(p+1,\ldots,p+q\right)\in S_{p+q}\textrm{.}\]
Let \(\varphi_{-}=\left(-1,\ldots,-p\right)\left(-\left(p+1\right),\ldots,-\left(p+q\right)\right)\) and \(\varphi=\varphi_{+}\varphi_{-}^{-1}\).  Let \(\alpha\in \mathrm{PM}_{\mathrm{nc}}\left(\varphi_{+}\varphi_{-}^{-1}\right)\) have no fixed points and no cycles containing neighbours \(a\), \(\varphi_{+}\left(a\right)\).  Then \(p=q\), and \(\alpha\) is a pairing, with pairs of the form \(\left(i,\varphi^{-i}\left(\pm\left(2p+1\right)-k\right)\right)\), where \(k\in\pm\left[p+1,2p\right]\) can be thought of as a signed offset (see Figure~\ref{figure: spokes}).

If \(\zeta_{+}=\left(1,\ldots,2p\right)\in S_{2p}\), then \(\chi\left(\zeta_{+},\alpha\right)\) depends on \(k\) as follows:
\[\chi\left(\zeta_{+},\alpha\right)=\left\{\begin{array}{ll}2\textrm{,}&k=p\\1\textrm{,}&k=-1\\0\textrm{,}&\textrm{otherwise}\end{array}\right.\textrm{.}\]
\label{lemma: spokes}
\end{lemma}
\begin{proof}
Since \(\alpha\in\mathrm{PM}_{\mathrm{nc}}\left(\varphi_{+}\varphi_{-}^{-1}\right)\), then for some choice of sign, it does not connect \(I:=\left[p\right]\cup\left(\pm\left[p+1,p+q\right]\right)\) to \(-I\).  Throughout the rest of this proof, we will implicitly consider \(\varphi\) and \(\alpha\) on \(I\).  Then \(\alpha\) is annular noncrossing on \(\varphi\) as described in Definition~\ref{definition: annular noncrossing}.

We show first that \(\alpha\) cannot contain any \(a\) and \(b\) in the same orbit of \(\varphi\).  We know there is at least one cycle of \(\alpha\) which contains an \(x,y\in I\) in distinct cycles of \(\varphi\), and \(\left.\alpha\right|_{I\setminus\left\{x,y\right\}}\in S_{\mathrm{disc-nc}}\left(\lambda_{x,y}\right)\) (the second and third annular crossing conditions imply the first and second disc crossing conditions, respectively, for cycles not containing \(x\) and \(y\), and the annular nonstandard conditions and the first annular crossing condition imply the first and second disc crossing conditions, respectively, when one of the cycles involved is the one containing \(x\) and \(y\)).  Then, following the proof of Lemma~\ref{lemma: singlets}, if a cycle of \(\alpha\) contains two elements from the same cycle of \(\varphi\), then a cycle of \(\alpha\) must contain either a singlet or two neighbours in \(\lambda_{x,y}\).

Thus \(\alpha\) must be a pairing on \(I\), in which each pair contains an element of \(\left[p\right]\) and an element of \(\pm\left[p+1,p+q\right]\) (so \(p=q\)).  Define a signed offset \(k\in\pm\left[p\right]\) by letting \(\pm\left(2p+1\right)-k\) be the partner of \(p\).  Then for \(i\in\left[p\right]\), \(i\) is partnered with \(\varphi^{-i}\left(\pm\left(2p+1\right)-k\right)\).  (Otherwise, let \(i\) be the first element of \(\left[p\right]\) not partnered with \(\varphi^{-i}\left(k\right)\); the partner of \(i\) will be \(\varphi^{-\left(i+m\right)}\left(\pm\left(2p+1\right)-k\right)\) for some \(0<m<p\).  Letting \(x=i-1\) and \(y=\varphi^{-\left(i-1\right)}\left(\pm\left(2p+1\right)-k\right)\), the third annular noncrossing condition with \(\lambda_{x,y}\) implies that the elements in \(\left\{\varphi^{-i}\left(\pm\left(2p+1\right)-k\right),\ldots,\varphi^{-\left(i+m-1\right)}\left(\pm\left(2p+1\right)-k\right)\right\}\) may not share cycles with elements outside that set.  Since they are all elements of the same cycle of \(\varphi\), they may also not share cycles with each other, a contradiction.)

We may calculate \(K\left(\zeta_{+},\alpha\right)\) for several cases.  If \(k=p\), then
\[\mathrm{FD}\left(K\left(\zeta_{+},\alpha\right)\right)=\left(1,2p-1\right)\left(2,2p-2\right)\cdots\left(p-1,p+1\right)\left(p\right)\left(2p\right)\]
so \(\chi\left(\zeta_{+},\alpha\right)=2\).  If \(k\in\left[p-1\right]\), then
\begin{multline*}
\mathrm{FD}\left(K\left(\zeta_{+},\alpha\right)\right)\\=\left(1,2p-k-1\right)\left(2,2p-k-2\right)\cdots\left(p-k-1,p+1\right)\left(p-k,p,2p-k,2p\right)\\\left(p-k+1,2p-1\right)\cdots\left(p-1,2p-k+1\right)
\end{multline*}
so \(\chi\left(\zeta_{+},\alpha\right)=0\).  If \(k=-1\), then \[\mathrm{FD}\left(K\left(\zeta_{+},\alpha\right)\right)=\left(1,-\left(p+1\right)\right)\left(2,-\left(p+2\right)\right)\cdots\left(p-1,-\left(2p-1\right)\right)\left(p,-2p\right)\]
so \(\chi\left(\zeta_{+},\alpha\right)=1\).  If \(k\in\left[-p,-2\right]\), then
\begin{multline*}
\mathrm{FD}\left(K\left(\zeta_{+},\alpha\right)\right)\\=\left(1,-\left(2p+k+2\right)\right)\left(2,-\left(2p+k+3\right)\right)\cdots\left(-k,-\left(2p-1\right)\right)\\\left(-k+1,-2p,2p+k+1,-p\right)\left(-k+2,-\left(p+1\right)\right)\cdots\left(p-1,-\left(2p+k\right)\right)
\end{multline*}
so \(\chi\left(\zeta_{+},\alpha\right)=0\).
\end{proof}

\begin{proposition}
Let \(A_{1},\ldots,A_{C}\) be independent sets of symplectically invariant quaternionic matrices with second-order limit distributions.

Let \(X_{1},\ldots,X_{p},Y_{1},\ldots,Y_{q}:\Omega\rightarrow M_{N\times N}\left(\mathbb{H}\right)\) be random quaternionic matrices with \(\mathbb{E}\left(\mathrm{tr}\left(X_{k}\right)\right)=\mathbb{E}\left(\mathrm{tr}\left(Y_{k}\right)\right)=0\) for all \(k\), and such that \(X_{k}\in A_{u\left(k\right)}\) and \(Y_{k}\in A_{v\left(k\right)}\), for words \(u:\left[p\right]\rightarrow\left[C\right]\) with either \(u\left(1\right)\neq u\left(2\right)\neq\cdots\neq u\left(p\right)\neq u\left(1\right)\) or \(p=1\) and \(v:\left[q\right]\rightarrow\left[C\right]\) with either \(v\left(1\right)\neq v\left(2\right)\neq\cdots\neq v\left(q\right)\neq v\left(1\right)\) or \(q=1\).

Then if \(p\neq q\), or if \(p=q=1\) and \(u\left(1\right)\neq v\left(1\right)\):
\[\lim_{N\rightarrow\infty}k_{2}\left(\mathrm{tr}\left(X_{1}\cdots X_{p}\right),\mathrm{tr}\left(Y_{1}\cdots Y_{q}\right)\right)=0\]
and if \(p=q>1\):
\begin{multline*}
\lim_{N\rightarrow\infty}k_{2}\left(\mathrm{tr}\left(X_{1}\cdots X_{p}\right),\mathrm{tr}\left(Y_{1}\cdots Y_{q}\right)\right)
\\=\prod_{i=1}^{p}\lim_{N\rightarrow\infty}\mathbb{E}\left(\mathrm{Re}\left(\mathrm{tr}\left(X_{i}Y_{p+1-i}\right)\right)\right)+\frac{1}{4}\sum_{k=1}^{p-1}\prod_{i=1}^{p}\lim_{N\rightarrow\infty}\mathbb{E}\left(\mathrm{Re}\left(\mathrm{tr}\left(X_{i}Y_{p+1-k-i}\right)\right)\right)
\\-\frac{1}{2}\prod_{i=1}^{p}\lim_{N\rightarrow}\mathbb{E}\left(\mathrm{Re}\left(\mathrm{tr}\left(X_{i}Y_{i}^{\ast}\right)\right)\right)+\frac{1}{4}\sum_{k=1}^{p-1}\prod_{i=1}^{p}\lim_{N\rightarrow\infty}\mathbb{E}\left(\mathrm{Re}\left(\mathrm{tr}\left(X_{i}Y_{k+i}^{\ast}\right)\right)\right)
\end{multline*}
where indices are taken modulo \(p\).
\label{proposition: spokes}
\end{proposition}
\begin{remark}
We note that if \(u\left(k\right)\neq v\left(l\right)\), then
\begin{multline*}
\mathbb{E}\left(\mathrm{Re}\left(\mathrm{tr}\left(X_{k}Y_{l}^{\left(\pm 1\right)}\right)\right)\right)=f_{u\left(k\right)}\left(\left(k\right)\left(-k\right)\right)f_{v\left(l\right)}\left(\left(l\right)\left(-l\right)\right)\\=\mathbb{E}\left(\mathrm{Re}\left(\mathrm{tr}\left(X_{k}\right)\right)\right)\mathbb{E}\left(\mathrm{Re}\left(\mathrm{tr}\left(Y_{l}\right)\right)\right)=0\textrm{,}
\end{multline*}
so any summand containing a term which matches matrices with \(u\left(k\right)\neq v\left(l\right)\) vanishes.
\end{remark}
\begin{proof}
Let
\[\varphi_{\mathrm{tr}}:=\left(\infty\right)\left(1,\ldots,p\right)\left(p+1,\ldots,p+q\right)\textrm{,}\]
let
\[\varphi_{\mathrm{Re}}:=\left(\infty,1,\ldots,p+q\right)\textrm{,}\]
and let \(w:\left[p+q\right]\rightarrow\left[C\right]\) be given by
\[w\left(k\right)=\left\{\begin{array}{ll}u\left(k\right)\textrm{,}&k\in\left[p\right]\\v\left(k-p\right)&k\in\left[p+1,p+q\right]\end{array}\right.\textrm{.}\]

If \(p=q=1\) and \(u\left(1\right)\neq v\left(1\right)\), then \(k_{2}\left(\mathrm{Tr}\left(X_{1}\right),\mathrm{Tr}\left(Y_{1}\right)\right)\) is the covariance of independent quantities and hence vanishes.

If \(p\neq q\) or \(p=q>2\), then applying Proposition~\ref{proposition: cumulants} to the left-hand side of the desired covariance in the statement of the proposition, we note that if \(\alpha\) does not connect the blocks of \(\pm\varphi\), then at least one of \(\left.\alpha\right|_{\left[p\right]}\) and \(\left.\alpha\right|_{\left[p+1,p+q\right]}\) must contain at least two single-element cycles (Lemma~\ref{lemma: singlets}).  If \(\pi\) connects the blocks of \(\pm\varphi\), then in an asymptotically nonvanishing term \(\pi\) must connect at most two blocks of \(\pm\alpha\) (since \(\chi\left(\alpha,\pi\right)\leq 2\#\left(\pm\alpha\vee\pi\right)\)), and \(\sigma=0_{p}\).  Thus a term of the form \(\mathbb{E}\left(\mathrm{Re}\left(\mathrm{tr}\left(X_{k}\right)\right)\right)=0\) or \(\mathbb{E}\left(\mathrm{Re}\left(\mathrm{tr}\left(Y_{k-p}\right)\right)\right)=0\) must appear in the expansion, where \(k\) is the element from a single-element cycle of \(\alpha\) which \(\pi\) does not connect to another cycle, and so the term vanishes.  Likewise, if \(\rho\) or \(\sigma\) connects the blocks of \(\pm\varphi\), then \(\pi\) must not connect any cycles of \(\alpha\), and there may not be more than one block of \(\sigma\) containing more than one element, so there must again be a term \(\mathbb{E}\left(\mathrm{Re}\left(\mathrm{tr}\left(X_{k}\right)\right)\right)=0\) or \(\mathbb{E}\left(\mathrm{Re}\left(\mathrm{tr}\left(Y_{k-p}\right)\right)\right)=0\) where \(k\) is the element from a single-element cycle of \(\pi\) which appears in a single-element block of \(\hat{\sigma}\).  Thus all terms in which \(\alpha\) does not connect the blocks of \(\pm\varphi\) vanish asymptotically.

In an asymptotically nonvanishing term in which \(\alpha\) connects the blocks of \(\pm\varphi\), \(\pi\) does not connect cycles of \(\alpha\), \(\rho=\left(\delta\alpha\right)\vee\left(\delta\pi\right)\), and \(\sigma=0_{p}\).  We calculate:
\begin{multline*}
\lim_{N\rightarrow\infty}k_{2}\left(\mathrm{Tr}\left(X_{1}\cdots X_{p}\right),\mathrm{Tr}\left(Y_{1}\cdots Y_{q}\right)\right)
\\=\sum_{\substack{\alpha\in \mathrm{PM}_{\mathrm{nc}}\left(\varphi_{\mathrm{tr}}\right)\\\pm\varphi\vee\alpha=1_{\pm\left[p+q\right]}}}\left(-2\right)^{\chi\left(\varphi_{\mathrm{Re}},\alpha\right)-2}\sum_{\pi\in S_{\mathrm{nc}}\left(\alpha\right):\pi\preceq\alpha}\left(\prod_{i\in\Lambda\left(\left(\delta\alpha\right)\vee\left(\delta\pi\right)\right)}\left(-1\right)^{i-1}C_{i-1}\right)\\\times\left(\mathbb{E}\circ\mathrm{Re}\circ\mathrm{tr}\right)_{\mathrm{FD}\left(\pi\right)}\left(X_{1},\ldots,X_{p},Y_{1},\ldots,Y_{q}\right)\textrm{.}
\end{multline*}

If \(\alpha\) contains any single-element cycle, then \(\mathbb{E}\left(\mathrm{Re}\left(\mathrm{tr}\left(X_{k}\right)\right)\right)=0\) or \(\mathbb{E}\left(\mathrm{Re}\left(\mathrm{tr}\left(Y_{k}\right)\right)\right)=0\) must appear for some \(k\).  Thus we may restrict our attention to \(\alpha\) which satisfy the hypotheses of Lemma~\ref{lemma: spokes}.

Since such \(\alpha\) are pairings, any \(\pi\preceq\alpha\) has cycles of length \(1\) or \(2\).  If any cycle of \(\pi\) has length \(1\) the term again vanishes.  Thus \(\pi=\alpha\) for any nonvanishing term.

For such a term, \(\delta\alpha=\delta\pi\), so \(\Lambda\left(\left(\delta\alpha\right)\vee\left(\delta\pi\right)\right)=\left(1,\ldots,1\right)\).  Thus the contribution of the product of Catalan numbers is \(1\).  Lemma~\ref{lemma: spokes} gives the possible values of \(\chi\left(\varphi_{\mathrm{tr}},\alpha\right)\), and hence the power of \(-2\), for the possible cases of \(\alpha\).
\end{proof}

This culminates in the following:
\begin{theorem}
Independent sets of symplectically invariant random quaternionic matrices are asymptotically quaternionic second-order free.
\end{theorem}
\begin{proof}
Asymptotic freeness is shown in Proposition~\ref{proposition: free}, second-order limit distribution follows from Corollary~\ref{corollary: cumulants}, and the formula for the asymptotic covariance of traces is given in Proposition~\ref{proposition: spokes}.
\end{proof}

\section{Specific matrix models}
\label{section: zoo}

We demonstrate in this section that four important matrix models have second-order limit distributions, so the results of this paper apply.  We consider upper bounds on cumulants of traces, which are of the form:
\begin{multline}
k_{m}\left(\mathrm{tr}\left(X_{1}^{\left(\varepsilon\left(1\right)\right)}\cdots X_{n_{1}}^{\left(\varepsilon\left(n_{1}\right)\right)}\right),\ldots,\mathrm{tr}\left(X_{n_{m-1}+1}^{\left(\varepsilon\left(n_{m-1}+1\right)\right)}\cdots X_{n_{m}}^{\left(\varepsilon\left(n_{m}\right)\right)}\right)\right)
\\=\sum_{\alpha\in\mathrm{PM}\left(n_{m}\right)}\left(-2\right)^{\chi\left(\zeta,\delta_{\varepsilon}\alpha\delta_{\varepsilon}\right)-2}N^{\chi\left(\varphi,\delta_{\varepsilon}\alpha\delta_{\varepsilon}\right)-2m}\\\times\sum_{\pi\in{\cal P}\left(m\right):\hat{\pi}\succeq\alpha}\mu\left(\pi,1_{m}\right)\prod_{V\in\hat{\pi}}f\left(\left.\alpha\right|_{V}\right)\textrm{.}
\label{formula: cumulant}
\end{multline}

We also give some formulas for the contribution of the four matrix models to the asymptotic covariance.  We may construct a centred element of an algebra of random matrices from any element by subtracting the expected value of its trace: \(X-\mathbb{E}\left(\mathrm{tr}\left(X\right)\right)\).  Then the factors appearing in the summands for the asymptotic covariance are of the form \begin{multline*}
\left(\mathbb{E}\circ\mathrm{Re}\circ\mathrm{tr}\right)\left(\left[X-\mathbb{E}\left(\mathrm{tr}\left(X\right)\right)\right]\left[Y-\mathbb{E}\left(\mathrm{tr}\left(Y\right)\right)\right]\right)\\=\mathbb{E}\left(\mathrm{Re}\left(\mathrm{tr}\left(XY\right)\right)\right)-\mathbb{E}\left(\mathrm{Re}\left(\mathrm{tr}\left(X\right)\right)\right)\mathbb{E}\left(\mathrm{Re}\left(\mathrm{tr}\left(Y\right)\right)\right)\textrm{.}
\end{multline*}
(The value of the function \(\mathbb{E}\circ\mathrm{Re}\circ\mathrm{tr}\) is unchanged when the terms in the product are cycled, and \(\mathbb{E}\) commutes with \(\mathrm{Re}\).)  Since matrix cumulants are asymptotically multiplictive, this is the sum over cumulants associated to diagrams which connect the intervals corresponding to \(X\) and \(Y\) (see Figure~\ref{figure: spoke}).  If \(X=Z_{1}^{\left(\varepsilon\left(1\right)\right)}\cdots Z_{m}^{\left(\varepsilon\left(m\right)\right)}\) and \(Y^{\left(\pm 1\right)}=Z_{m+1}^{\left(\varepsilon\left(m+1\right)\right)}\cdots Z_{m+n}^{\left(\varepsilon\left(m+n\right)\right)}\) where \(Z_{1},\ldots,Z_{m+n}\) have matrix cumulants \(f:\mathrm{PM}\left(m+n\right)\rightarrow\mathbb{C}\), then denoting \(\zeta=\left(1,\ldots,m+n\right)\left(-\left(m+n\right),\ldots,-1\right)\) and \(\pi=\left\{\left[m\right],\left[m+1,m+n\right]\right\}\):
\begin{multline*}
\lim_{N\rightarrow\infty}\left(\mathbb{E}\left(\mathrm{Re}\left(\mathrm{tr}\left(XY\right)\right)\right)-\mathbb{E}\left(\mathrm{Re}\left(\mathrm{tr}\left(X\right)\right)\right)\mathbb{E}\left(\mathrm{Re}\left(\mathrm{tr}\left(Y\right)\right)\right)\right)
\\=\sum_{\alpha\in\mathrm{PM}_{\mathrm{nc}}\left(\zeta\right):\pm\pi\vee \alpha=1_{\pm\left[m+n\right]}}\lim_{N\rightarrow\infty}f\left(\delta_{\varepsilon}\alpha\delta_{\varepsilon}\right)\textrm{.}
\end{multline*}

\begin{example}
If, in the calculation of \(\lim_{N\rightarrow\infty}k_{2}\left(X_{1}\cdots X_{p},Y_{1}\cdots Y_{p}\right)\), we encounter a term \(X_{k}=W_{1}\cdots W_{6}\) paired with a term \(Y_{l}^{\ast}\) where \(Y_{l}=W_{7}\cdots W_{10}\), we consider connected, noncrossing diagrams such as the one in Figure~\ref{figure: spoke}, which we can think of as an enlargement of one of the spokes in the diagrams in Figure~\ref{figure: spokes}.

\begin{figure}
\label{figure: spoke}
\centering
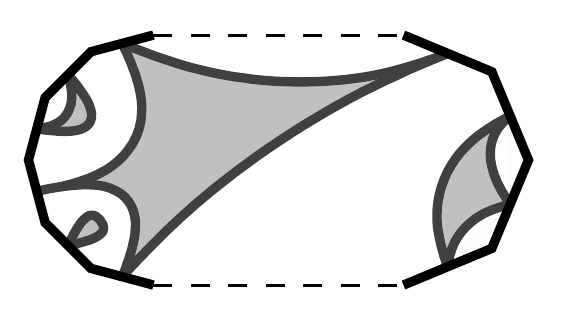
\caption{An enlargement of a term from a spoke between terms \(Y_{k}=W_{1}\cdots W_{6}\) and \(Y_{l}^{\ast}=W_{10}^{\ast}\cdots W_{7}^{\ast}\).}
\end{figure}

The diagram in Figure~\ref{figure: spoke} contributes
\[\mathrm{Re}\left(\mathrm{tr}\left(D_{1}D_{3}D_{6}D_{10}^{\ast}\right)\right)\mathrm{Re}\left(\mathrm{tr}\left(D_{2}\right)\right)\mathrm{Re}\left(\mathrm{tr}\left(D_{4}D_{5}\right)\right)\mathrm{Re}\left(\mathrm{tr}\left(D_{7}^{\ast}D_{9}^{\ast}D_{8}^{\ast}\right)\right)\textrm{.}\]
\end{example}

Throughout, a standard quaternionic Gaussian random variable is a random variable \(\xi_{0}+\xi_{1}i+\xi_{2}j+\xi_{3}k\) where the \(\xi_{i}\) are \(N\left(0,\frac{1}{4}\right)\) random variables and \(G:\Omega\rightarrow M_{M\times N}\left(\mathbb{H}\right)\) is a random \(M\times N\) matrix whose entries are i.i.d.\ standard quaternionic Gaussian random variables.  The matrix cumulants are calculated in \cite{2014arXiv1412.0646R}.

\subsection{Quaternionic Ginibre matrices}

Ginibre matrices are first defined in \cite{MR0173726}.
\begin{definition}
A {\em quaternionic Ginibre matrix} is a random matrix \(Z=\frac{1}{\sqrt{N}}G\) with \(M=N\).
\end{definition}
The matrix cumulants of Ginibre matrices are
\[f\left(\alpha\right)=\left\{\begin{array}{ll}1\textrm{,}&\textrm{\(\alpha\in{\cal P}_{2}\left(\pm\left[n\right]\right)\), \(\mathrm{sgn}\left(\alpha\left(k\right)\right)=-\mathrm{sgn}\left(k\right)\) for all \(k\)}\\0\textrm{,}&\textrm{otherwise}\end{array}\right.\textrm{.}\]

Since the cumulants are multiplicative,
\[\sum_{\pi\in{\cal P}\left(m\right):\hat{\pi}\succeq\alpha}\mu\left(\pi,1_{m}\right)\prod_{V\in\hat{\pi}}f\left(\left.\alpha\right|_{V}\right)=\left\{\begin{array}{ll}1\textrm{,}&\pm\varphi\vee\alpha=1_{\pm n_{m}}\\0\textrm{,}&\textrm{otherwise}\end{array}\right.\]
so in any non-vanishing term of (\ref{formula: cumulant}), \(\chi\left(\varphi,\alpha\right)\leq 2\).  Thus the \(\ast\)-algebra generated by a Ginibre matrix has a second-order limit distribution.

A noncrossing pairing must connect elements with oppositive values of \(\varepsilon\).  The number of such noncrossing pairings depends on the values of \(\varepsilon\).

\subsection{Gaussian symplectic ensemble matrices}

\begin{definition}
A {\em Gaussian symplectic ensemble} matrix, or GSE matrix, is a matrix \(\frac{1}{\sqrt{2N}}\left(G+G^{\ast}\right)\) with \(M=N\).
\end{definition}
The matrix cumulants of GSE matrices are
\[f\left(\alpha\right)=\left\{\begin{array}{ll}1\textrm{,}&\alpha\in{\cal P}_{2}\left(n\right)\\0\textrm{,}&\textrm{otherwise}\end{array}\right.\textrm{.}\]

Like Ginibre matrices, cumulants are multiplicative, so GSE matrices have a limit distribution.

Noncrossing pairings on the disc are enumerated by the Catalan numbers.  If \(X=T^{m}-\mathbb{E}\left(\mathrm{Re}\left(\mathrm{tr}\left(T^{m}\right)\right)\right)\) and \(Y=T^{n}-\mathbb{E}\left(\mathrm{Re}\left(\mathrm{tr}\left(T^{n}\right)\right)\right)\), then the contribution of the spoke is
\[\lim_{N\rightarrow\infty}\mathbb{E}\left(\mathrm{Re}\left(\mathrm{tr}\left(XY\right)\right)\right)=\left\{\begin{array}{ll}C_{\left(m+n\right)/2}-C_{m}C_{n}\textrm{,}&\textrm{\(m\), \(n\) even}\\C_{\left(m+n\right)/2}\textrm{,}&\textrm{\(m\), \(n\) odd}\\0\textrm{,}&\textrm{otherwise}\end{array}\right.\textrm{.}\]

\subsection{Quaternionic Wishart matrices}

Wishart matrices are first defined in \cite{Wishart}.

\begin{definition}
\label{definition: Wishart}
A {\em quaternionic Wishart} matrix is a matrix \(W=\frac{1}{N}G^{\ast}DG\), where \(D\in M_{M\times M}\left(H\right)\) is a deterministic matrix.  We may also consider a set of Wishart matrices with different deterministic matrices \(D_{k}\) but the same (not independent) \(G\): \(W_{1}=\frac{1}{N}G^{\ast}D_{1}G,\ldots,W_{n}=\frac{1}{N}G^{\ast}D_{n}G\).
\end{definition}
The matrix cumulant of \(W_{1},\ldots,W_{n}\) is
\[f\left(\alpha\right)=\left(\mathrm{Re}\circ\mathrm{tr}\right)_{\mathrm{FD}\left(\alpha^{-1}\right)}\left(D_{1},\ldots,D_{n}\right)\textrm{.}\]

As with Ginibre and GSE matrices, the cumulant is multiplicative, so Wishart matrices have a second-order limit distribution.

The value of an expression with Wishart matrices depends on the values of the \(D_{k}\).  If \(D_{k}=I_{M}\) for all \(k\), then it may be expressed in terms of the number of noncrossing hypermaps, which are counted by Catalan numbers.  In this case, if \(X=W_{1}\cdots W_{m}-\mathbb{E}\left(\mathrm{tr}\left(W_{1}\cdots W_{m}\right)\right)\) and \(Y=W_{m+1}\cdots W_{m+n}-\mathbb{E}\left(\mathrm{tr}\left(W_{m+1}\cdots W_{m+n}\right)\right)\), then
\[\mathbb{E}\left(\mathrm{Re}\left(\mathrm{tr}\left(XY\right)\right)\right)=C_{m+n}-C_{m}C_{n}\textrm{.}\]

\subsection{Haar-distributed symplectic matrices}

\begin{definition}
A {\em symplectic matrix} is a matrix \(U\in M_{N\times N}\left(\mathbb{H}\right)\) such that \(UU^{\ast}=I_{N}\).  The symplectic matrices \(\mathrm{Sp}\left(N\right)\) form a compact group, which has a Haar measure of finite measure.  When this measure is normalized to total measure \(1\), it is a probability measure.  A Haar-distributed random symplectic matrix is a matrix distributed according to this probability measure.
\end{definition}
The matrix cumulants of Haar-distributed symplectic matrices are
\[f\left(\alpha\right)=\left\{\begin{array}{ll}\mathrm{wg}\left(\Lambda\left(\mathrm{FD}\left(\alpha\right)\right)\right)\textrm{,}&\mathrm{sgn}\left(\alpha\left(k\right)\right)=-\mathrm{sgn}\left(k\right)\\0\textrm{,}&\textrm{otherwise}\end{array}\right.\textrm{.}\]

As calculated in Lemma~\ref{lemma: Weingarten}
\[\sum_{\pi\in{\cal P}\left(m\right):\hat{\pi}\succeq\alpha}\mu\left(\pi,1_{m}\right)\prod_{V\in\hat{\pi}}f\left(\left.\alpha\right|_{V}\right)=O\left(N^{2-\#\left(\pm\varphi\vee\alpha\right)}\right)\]
which, combined with the exponent on \(N\) in (\ref{formula: cumulant}) (where \(\chi\left(\varphi,\delta_{\varepsilon}\alpha\delta_{\varepsilon}\right)\leq 2\#\left(\pm\varphi\vee\alpha\right)\)) gives us that the \(m\)th cumulant of unnormalized traces is \(O\left(N^{2-m}\right)\), so Haar-distributed symplectic matrices have a second-order limit distribution.

Since \(U^{\ast}=U^{-1}\), any product of \(U\) and \(U^{\ast}\) is of the form \(U^{n}\), \(n\in\mathbb{Z}\).  If \(X=U^{m}-\mathbb{E}\left(\mathrm{tr}\left(U^{m}\right)\right)\) and \(Y=U^{n}-\mathbb{E}\left(\mathrm{tr}\left(U^{n}\right)\right)\), then since a noncrossing \(\alpha\) must take any \(k\) to an \(\alpha\left(k\right)\) such that \(\varepsilon\left(\alpha\left(k\right)\right)=-\varepsilon\left(k\right)\) (i.e.\ from the other term), \(m=-n\) and:
\[\lim_{N\rightarrow\infty}\mathbb{E}\left(\mathrm{Re}\left(\mathrm{tr}\left(XY\right)\right)\right)=\delta_{m,-n}\textrm{.}\]

\bibliography{paper}
\bibliographystyle{plain}

\end{document}